\documentclass[12pt,oneside]{amsart}
\usepackage{amsbsy,mathrsfs}
\usepackage{amstext,amssymb}
\usepackage{amsthm}
\usepackage{mathabx}
\usepackage{stmaryrd}
\usepackage{bbm}
\usepackage[T1]{fontenc}
\usepackage[latin9]{inputenc}
\usepackage{geometry}
\usepackage{mathrsfs}
\usepackage{graphicx,epstopdf}
\usepackage{fixmath}
\usepackage[colorlinks,bookmarks,linkcolor=black,citecolor=black]{hyperref}
\usepackage{enumitem}
\usepackage{url}
\usepackage{soul}

\makeatletter
\newcommand{\By}[2]{\overset{\mbox{\tiny{#1}}}{#2}}
\newcommand{\ByRef}[2]{   \By{\eqref{#1}}{#2} }

\newcommand{\lBy}[1]{     \By{#1}{<} }

\newcommand{\leBy}[1]{    \By{#1}{\le} }
\newcommand{\geBy}[1]{    \By{#1}{\ge} }
\newcommand{\eqByRef}[1]{ \ByRef{#1}{=} }
\newcommand{\lByRef}[1]{  \ByRef{#1}{<} }

\newcommand{\leByRef}[1]{ \ByRef{#1}{\le} }
\newcommand{\geByRef}[1]{ \ByRef{#1}{\ge} }

\newcommand{\LandauOmega}{\mathsf{\Omega}}

\newcommand{\LandauO}{\mathsf{O}}
\newcommand{\Landauo}{\mathsf{o}}

\newtheorem{thm}{Theorem}[section]
\newtheorem{lem}[thm]{Lemma}
\newtheorem{proposition}[thm]{Proposition}

\newtheorem{corollary}[thm]{Corollary}

\newtheorem{defn}[thm]{Definition}
\newtheorem{remark}[thm]{Remark}

\newtheorem{Prop26}{Prop26}[section]
\newtheorem{claim26}{Claim}[Prop26]

\newtheorem{claim}{Claim}
\numberwithin{claim}{thm}

\newcommand{\oldqed}{}
\newcommand{\qedClaim}{\hfill\scalebox{.6}{$\Box$}}
\renewcommand{\epsilon}{\varepsilon}
\newenvironment{claimproof}[1][Proof]{
  \renewcommand{\oldqed}{\qedsymbol}
  \renewcommand{\qedsymbol}{\qedClaim}
  \begin{proof}[#1]
}{
  \end{proof}
  \renewcommand{\qedsymbol}{\oldqed}
}
\newcommand{\T}{\mathcal{T}}
\renewcommand{\O}{\mathcal{O}}

\newcommand{\EXP}{\mathbb{E}}
\newcommand{\PROB}{\mathbb{P}}

\newcommand{\stab}{\mathrm{Stab}}

\newcommand{\lr}{\leftrightarrow}
\newcommand{\FREQ}{\mathrm{Freq}}
\newcommand{\congpure}{\stackrel{\mathrm{pure}}{\cong}}

\newcommand{\RESISTANCE}{\mathcal{R}_{\mathrm{eff}}}

\def\parent{\mathrm{Par}}

\newcommand{\POISSON}{\mathsf{Poisson}}
\newcommand{\BIN}{\mathsf{Bin}}

\newcommand{\bN}{\mathbb{N}}
\newcommand{\bR}{\mathbb{R}}

\newcommand{\cC}{\mathcal{C}}

\newcommand{\eps}{\varepsilon}
\renewcommand{\rho}{\varrho}
\renewcommand{\subset}{\subseteq}

\begin{document}

\title{The local limit of the uniform spanning tree on~dense~graphs}

\author{Jan Hladk\'y}
\address{Institut f\"ur Geometrie, TU Dresden, 01062 Dresden, Germany.}
\email{honzahladky@gmail.com}
\author{Asaf Nachmias}
\address{Department of Mathematical Sciences, Tel Aviv University, Tel Aviv 69978, Israel.}
\email{asafnach@post.tau.ac.il}
\author{Tuan Tran}
\address{Department of Mathematics, ETH, 8092 Zurich, Switzerland.}
\email{manh.tran@math.ethz.ch}

\begin{abstract}
Let $G$ be a connected graph in which almost all vertices have linear degrees and let $\T$ be a uniform spanning tree of $G$. For any fixed rooted tree $F$ of height $r$ we compute the asymptotic density of vertices $v$ for which the $r$-ball around $v$ in $\T$ is isomorphic to $F$.
We deduce from this that if $\{G_n\}$ is a sequence of such graphs converging to a graphon $W$, then the uniform spanning tree of $G_n$ locally converges to a multi-type branching process defined in terms of $W$.

As an application, we prove that in a graph with linear minimum degree, with high probability, the density of leaves in a uniform spanning tree is at least $e^{-1}-\Landauo(1)$, the density of vertices of degree $2$ is at most $e^{-1}+\Landauo(1)$ and the density of vertices of degree $k\geq 3$ is at most ${(k-2)^{k-2} \over (k-1)! e^{k-2}} + \Landauo(1)$. These bounds are sharp.
\end{abstract}

\maketitle

\vspace{-.5in}

\section{Introduction}

It is a classical fact \cite{Gri:RandomTree,Kolchin1977} that a uniformly chosen tree from the set of $n^{n-2}$ trees on $n$ vertices, viewed from an independently chosen uniform random vertex, is locally distributed as a $\POISSON(1)$ Galton--Watson branching process conditioned to live forever when $n$ is large. One can use this result to find the distribution of a certain ``local'' structures. For instance, it follows that the degree distribution of a uniformly chosen vertex of a uniformly chosen tree on $n$ vertices converges to the law of a $\POISSON(1)+1$ random variable.

A uniformly chosen tree on $n$ vertices is a uniform spanning tree (UST) of the complete graph on $n$ vertices. Our goal in this paper is to explicitly describe the local structure of the UST of any dense graph or, equivalently, of a sequence of dense graphs converging to a given graphon. Let us first present this result.

Given a connected graph $G$ we write $\T$ for a uniformly drawn spanning tree of $G$ and $B_\T(v,r)$ for the graph-distance ball of radius $r$ in $\T$ around the vertex $v\in V(G)$. Our goal is to describe the asymptotic distribution of $B_\T(X,r)$, viewed up to graph isomorphism, where $X$ is a uniformly chosen random vertex of $G$. To that aim, let $\Omega=[0,1]$ and $\mu$ be the Lebesgue measure on $\Omega$. For a given graphon $W:\Omega^2\rightarrow[0,1]$ (see Section~\ref{sec:graphonpreliminaries} for a brief introduction to graphons) and $\omega \in \Omega$ we write $\deg(\omega)=\int_{y \in \Omega} W(\omega,y)\mathrm{d}y$, and call this number the {\bf degree} of $\omega$. A graphon is called {\bf nondegenerate} if for almost every $\omega\in\Omega$ we have $\deg(\omega)>0$.

Let $T$ be a fixed rooted tree with $\ell \geq 2$ vertices and of height $r\geq 1$. We write by $\stab_T$ the set of graph automorphisms of $T$ that preserve the root. In what follows we denote the vertices of $T$ by the numbers $\{1,\ldots, \ell\}$ such that the vertices $p, \ldots, \ell$ are the vertices at height $r$ of $T$ and $p\in \{2,\ldots,\ell\}$. Given a nondegenerate graphon $W$ and a tree $T$ as above we define
\begin{equation}
\label{eq:frewTW}
\FREQ(T;W):= {1 \over |\stab_T|} \int\displaylimits_{\omega_1, \ldots, \omega_\ell} \exp \left(-\sum_{j=1}^{p-1} b_W(\omega_j) \right) \frac{\sum_{j=p}^\ell \deg(\omega_j) }{\prod_{j=1}^\ell \deg(\omega_j)} \prod_{\substack{(i,j) \in E(T)}} W(\omega_i, \omega_j) \mathrm{d}\omega_1 \cdots \mathrm{d}\omega_\ell \, ,
\end{equation}
where
\begin{equation}\label{eq:defbb}
b_W(\omega) = \int_{y\in \Omega} {W(\omega,y) \over \deg(y)} \mathrm{d}\mu\;.
\end{equation}

\begin{thm}\label{thm:mainthm2} Let $T$ be a fixed rooted tree as above, and $W:\Omega^2 \rightarrow [0,1]$ be a nondegenerate graphon. Then for any $\eps>0$ there exists $\xi=\xi(\eps,W,T)>0$ such that if $G$ is a connected simple graph on at least $\xi^{-1}$ vertices that is $\xi$-close to $W$ in the cut-distance, then with probability at least $1-\eps$ a uniformly chosen spanning tree $\T$ of $G$ satisfies
	$$ \big | \PROB( B_\T(X,r) \cong T ) - \FREQ(T;W) \big | \leq \eps \, ,$$
	where $X$ is an independently and uniformly chosen vertex of $G$, and by $B_\T(X,r) \cong T$ we mean that between the two rooted trees there is a graph-isomorphism preserving the root.
\end{thm}

Theorem~\ref{thm:mainthm2} is a natural statement in the context of limits of graph sequences. It asserts that if $\{G_n\}$ is a sequence of connected graphs converging to a graphon $W$, then the UST of $G_n$ converges locally to a certain multi-type branching process that is defined in terms of $W$. This interpretation involve two graph limit procedures: a dense graph limit \cite{Lovasz2006,Borgs2008c} to describe the limit of the dense graph sequence $G_n$, and sparse graph limit to describe the random limit of the UST of $G_n$, known as Benjamini--Schramm convergence \cite{BeSc}. We refer the reader to Sections~\ref{sec:graphonpreliminaries} and \ref{sec:sparselimits} for an introduction to these limiting procedure, and begin by describing the limiting branching process.

\begin{defn}\label{def:branchingprocess}
	Given a nondegenerate graphon $W:\Omega^2\rightarrow[0,1]$ we define a multi-type branching process $\kappa_W$. The process has continuum many types that are either $(\mathsf{anc},\omega)$ or $(\mathsf{oth},\omega)$, where $\omega\in\Omega$. Here, ``$\mathsf{anc}$'' stands for ``ancestral'' and ``$\mathsf{oth}$'' stands for ``other''.
	\begin{enumerate}
		\item The initial particle (i.e., the root) has type $(\mathsf{anc},\omega)$, where the distribution of $\omega$ is $\mu$.
		\item If a particle has type $(\mathsf{oth},\omega)$, then its progeny is $\{(\mathsf{oth},\omega_1),\ldots,(\mathsf{oth},\omega_k))\}$ where $\{\omega_1,\ldots,\omega_k\}$ are a Poisson point process on $\Omega$ with intensity $\frac{W(\omega,\omega')}{\deg(\omega')}$ at $\omega'\in\Omega$. In particular, $k$ has distribution $\POISSON\left(b_W(\omega)\right)$, where $b_W(\cdot)$ is defined in~\eqref{eq:defbb}.
		\item If a particle has type $(\mathsf{anc},\omega)$ the progeny is $\{(\mathsf{anc},\omega_0),(\mathsf{oth},\omega_1),\ldots,(\mathsf{oth},\omega_k))\}$ where $\{\omega_1,\ldots,\omega_k\}$ are a Poisson point process on $\Omega$ with the same intensity as above and $\omega_0$ is an independent new particle of $\Omega$ which is distributed according to the probability measure on $\Omega$ that has density ${W(\omega,\cdot) \over \deg(\omega)}$.
	\end{enumerate}
\end{defn}
We note that an ancestral vertex (i.e., of type $\mathsf{anc}$) always has at least $1$ progeny and since the initial particle is ancestral the process $\kappa_W$ survives forever with probability $1$. The following theorem is quickly deduced from Theorem~\ref{thm:mainthm2} at Section~\ref{sec:Theorem1Theorem2}.

\begin{thm}\label{thm:main}
	Suppose that $\{G_n\}$ is a sequence of simple connected graphs of growing orders that converge to a nondegenerate graphon $W$ in the cut-distance and let $T_n$ be a UST of $G_n$. Then the sequence $\{T_n\}$ almost surely converge in the Benjamini--Schramm sense to $\kappa_W$.
\end{thm}

Since the $L^1$-norm (corresponding to the edit distance in graph theory) is finer than the cut-norm, we in particular obtain the following. Suppose that $G$ is a large connected simple $n$-vertex graph with minimum degree $\LandauOmega(n)$. Suppose that we add and/or delete $\Landauo(n^2)$ edges in a way that the resulting graph $G'$ stays connected. Then the structure of a typical UST of $G$ and of $G'$ is very similar in the Benjamini--Schramm sense. Even this statement seems to be new.

When $G_n$ are a sequence of dense \emph{regular} graphs, then the function $b_W:\Omega\to (0,\infty)$ defined in~\eqref{eq:defbb} is identically $1$ and we obtain the following.

\begin{corollary} Suppose that $\{G_n\}$ is a sequence of simple connected regular graphs of growing orders that converge to a nondegenerate graphon $W$ in the cut-distance and let $T_n$ be a UST of $G_n$. Then the sequence $\{T_n\}$ almost surely converge in the Benjamini--Schramm sense to a $\POISSON(1)$ Galton--Watson branching process conditioned to live forever.
\end{corollary}

\medskip
Theorem~\ref{thm:main} allows us to deduce several extremal properties of the UST on dense graphs. The number of vertices of degree $k\geq 1$ in the UST of the complete graph $K_n$ is $(e^{-1}/(k-1)!+\Landauo(1))n$. In fact, using Pr\"ufer codes one can establish that the degree of the a vertex in the UST of $K_n$ has distribution $1+\BIN(n-2,\frac{1}n)\approx 1+\POISSON(1)$.\footnote{Pr\"ufer codes, see e.g. \cite[page 245]{MR2469243}, provide a standard bijection between spanning trees of $K_n$ and and words of length $n-2$ over the alphabet $V(K_n)$. A quick look at this bijection shows the number of occurrences of any letter $v\in V(K_n)$ in a Pr\"ufer code is the degree of the vertex $v$ in the corresponding spanning tree decreased by~1. Therefore, the degree of $v$ in a uniform spanning tree has indeed  distribution $1+\BIN(n-2,\frac{1}n)$.} Using Theorem \ref{thm:mainthm2} we are able to find the extremal values of the number of vertices of degree $k$ in a general dense graph. In the following theorem we show that the complete graph (or any other \emph{regular} dense graph) is the minimizer of the number of leaves and the maximizer of the number of vertices of degrees $2$ and $3$ among the class of dense graphs. Somewhat surprisingly, the maximizer for the number of vertices of degree $k\geq 4$ is a different dense graph (see Section~\ref{sec:degreeextremal}).

\begin{thm}\label{cor.degdist} For any $k \geq 1$ we denote by $L_k$ the random variable counting the number of vertices of degree $k$ in a UST of a simple connected graph $G$. For every $\epsilon,\delta>0$ there exist numbers $n_0\in\bN$ and $\gamma>0$ such that the following holds: Whenever $G$ is a graph on $n\geq n_0$ vertices with at least $(1-\gamma)n$ vertices of degrees at least $\delta n$ then
	\begin{equation}\label{eq:extdeg1}
	\PROB\big ( L_1 \leq (e^{-1}-\epsilon)n \big ) \leq \epsilon \, ,
	\end{equation}
	\begin{equation}\label{eq:extdeg2}
	\PROB\big ( L_2 \geq (e^{-1}+\epsilon)n \big ) \leq \epsilon \, ,
	\end{equation}
	and for any $k \geq 3$ we have
	\begin{equation}\label{eq:extdegbiggerthan2}
	\PROB\left ( L_k \geq \left({1 \over (k-1)!}\tfrac{(k-2)^{k-2}}{e^{k-2}}+\epsilon\right)n\right ) \leq \epsilon\;.
	\end{equation}
\end{thm}

We derive Theorem~\ref{cor.degdist} in Section~\ref{sec:degreeextremal}.
The rest of this section is organized as follows. We first give the formal definitions and background of dense and sparse graph limits necessary for the reader to parse the statement of Theorems \ref{thm:mainthm2} and \ref{thm:main}. Next we discuss some aspects of the theorem, such as the necessity of its assumptions. We end this section with a discussion of related results in the literature.

\subsection{Dense graph limits} \label{sec:graphonpreliminaries}

Graphons were introduced by Borgs, Chayes, Lov\'asz, S\'os, Szegedy, and Vesztergombi~\cite{Lovasz2006,Borgs2008c} as limit objects to sequences of dense graphs. Here, we review basic facts and we refer the reader to~\cite[Part~II]{Lovasz2012} for a thorough treatment of the subject. A {\bf graphon} is a symmetric Lebesgue measurable function $W:\Omega^2\rightarrow [0,1]$, where $\Omega$ is any standard atomless probability space.\footnote{Recall that all such probability spaces are isomorphic.} The underlying measure on $\Omega$ will be always denoted by $\mu$.

Given a measurable function $U:\Omega^2\rightarrow[-1,1]$ we define its {\bf cut-norm} by
\begin{equation}
\label{eq:cutnormdef}
\|U\|_\square=\sup_{S,T}\left|\int_{x\in S}\int_{y\in T}U(x,y)\right|,
\end{equation}
where $S$ and $T$ range over all measurable subsets of $\Omega$. Now, we can define the key notion of {\bf cut-distance}. For two graphons $W_1,W_1:\Omega^2\rightarrow[0,1]$, we define
\begin{equation}\label{eq:defcutdist}
\delta_\square(W_1,W_2)=\inf_{\varphi} \|W_1^\varphi-W_2\|_\square\;,
\end{equation} where $\varphi:\Omega\rightarrow \Omega$ ranges through all measure preserving automorphisms of $\Omega$, and $W_1^\varphi$ stands for a graphon defined by $W_1^\varphi(x,y)=W_1(\varphi(x),\varphi(y))$. Note that the definition of $\delta_\square(W_1,W_2)$ extends in a straightforward way if $W_2$ lives on some other standard atomless probability space $\Lambda$. In that case, $\varphi$ ranges of all measure preserving bijections from $\Lambda$ to $\Omega$.

Suppose that $G$ is an $n$-vertex graph. Then we can consider a {\bf graphon representation} of $G$. To this end, partition an atomless standard probability space $\Omega$ into $n$ sets, each set of measure $\frac1n$, $\Omega=\bigsqcup_{v\in V(G)} \Omega_v$. Then define a graphon $W_G$ to be $1$ on  $\Omega_u\times \Omega_v$ for each edge $uv\in E(G)$, and $0$ otherwise. Note that $W_G$ is not unique as it depends on the choice of the partition $\{\Omega_v\}$. With the notion of a graphon representation, we can define distance between a graph and a graphon. Namely, if $W:\Omega^2\rightarrow [0,1]$ is a graphon and $G$ is a graph, we define $\delta_\square(W,G):=\delta_\square(W,W_G)$. This definition does not depend on the choice of the representation $W_G$. We say that $G$ is {\bf $\alpha$-close} to $W$ if $\delta_\square(W,G)\le \alpha$. Though through much of the paper, we shall work with loopless multigraphs (introduced in Section~\ref{ssec:densexpanders}), we always restrict ourselves to graphs when representing as graphons, or when referring to the cut-distance.

Throughout the paper, all sets and functions are tacitly assumed to be measurable. Conversely all our constructions of auxiliary sets and functions are measurable, too.

We say that a sequence $(G_n)_n$ of simple graphs {\bf converges} to a graphon $W$ if and only if $\delta_\square(W,G_n) \to 0$. Now, it is possible to understand the first sentence of Theorem~\ref{thm:main}. The following statement, proved first in~\cite{Lovasz2006}, is the core of the theory of graphons.
\begin{thm}\label{thm:graphsconverge}
	For every sequence of simple graphs of increasing orders there exists a subsequence which converges to a graphon.
\end{thm}

\subsection{Sparse graph limits}\label{sec:sparselimits} Let $\mathcal{G}_\bullet$ denote the space of rooted locally finite graphs viewed up to root-preserving graph isomorphisms. That is, each element of $\mathcal{G}_\bullet$ is $(G,\rho)$ where $G$ is a graph and $\rho$ is a vertex of it, and two such elements $(G_1,\rho_1)$ and $(G_2,\rho_2)$ are considered equivalent if and only if there exists a graph automorphism $\varphi:G_1 \to G_2$ such that $\varphi(\rho_1)=\rho_2$. Given a rooted graph $(G,\rho)$ and an integer $r \geq 1$ we write $B_G(\rho,r)$ for the graph-distance ball of radius $r$ around $\rho$ in $G$, that is, $B_G(\rho,r)\in \mathcal{G}_\bullet$ is a finite graph rooted at $\rho$ on the set of vertices of graphs distance at most $r$ from $\rho$ in $G$ together with all the edges induced from $G$. There is a natural notion of a metric on $\mathcal{G}_\bullet$. The distance between two elements $(G_1,\rho_1), (G_2, \rho_2) \in \mathcal{G}_\bullet$ is defined to be $2^{-R}$ where $R\geq 0$ is the largest number such that there is a root-preserving graph isomorphism between $B_{G_1}(\rho_1,R)$ and $B_{G_2}(\rho_2,R)$. Having defined the metric we can consider probability spaces on $\mathcal{G}_\bullet$ with respect to the Borel $\sigma$-algebra.

We say that the law of a random element $(G,\rho)\in \mathcal{G}_\bullet$ is the {\bf Benjamini--Schramm limit} of a (possibly random) sequence of finite graphs $G_n$, if and only if, for any fixed integer $r \geq 1$ the random variable $B_{G_n}(\rho_n,r)$ converges in distribution to $B_G(\rho,r)$ where $\rho_n$ is an independently chosen uniform random vertex of $G_n$. In this case we say that the sequence $G_n$ Benjamini--Schramm converges to $(G,\rho)$. Note that by putting $r=1$ in the definition we deduce that in this convergence the degree of the random root $\rho_n$ must converge to the degree of $\rho$ which explains why this limiting procedure is best suited for sparse graphs. See further discussion in~\cite{BeSc}. \medskip

We have finished defining all the needed terminology required to parse Theorems~\ref{thm:mainthm2} and~\ref{thm:main}.

\subsection{Necessity of the assumptions}

The assumptions in Theorem~\ref{thm:main} are the minimal necessary. First, we obviously need the assumption that $G_n$ are connected in order for spanning trees to exist.

Next we claim that the local structure of a uniform spanning tree cannot be determined from a degenerate graphon $W$. Indeed, suppose that a graphon $W$ is given, and let $\Omega^0$ be the elements of $\Omega$ that have zero degree $W$ and $\Omega^+ = \Omega \setminus \Omega^0$. Assume that $W$ is degenerate so that $\mu(\Omega^0)=\delta>0$.

We can now construct two graph sequences that converge to $W$. We start with dense graphs $G_n$ of size $(1-\delta)n$ that converge\footnote{Such a sequence can be obtained for example by taking typical inhomogeneous random graphs $\mathbb G(n,W^+)$, see \cite[Lemma 10.16]{Lovasz2012}.} to $W^+=W_{\mid \Omega^+}$  and in the first sequence we attach to $G_n$ a path of length $\delta n$ at an arbitrary vertex and in the second sequence we attach $\delta n$ edges arbitrary to an vertex of $G_n$ creating $\delta n$ new vertices of degree $1$. It is clear that both sequences converge to $W$. However, the USTs on the two sequences have different Benjamini--Schramm limits. Indeed, let $p_1$ denote the probability that in $\kappa_{W^+}$ the root is a leaf. Then the probability that a randomly chosen vertex is a leaf in the first sequence tends to $(1-\delta)p_1$ and in the second sequence this probability tends to $(1-\delta)p_1 + \delta$.

\subsection{Discussion}

Theorem~\ref{thm:main} shows that the local structure of the UST is continuous on the space of dense graphs with the cut-metric \eqref{eq:defcutdist} and describes this local structure explicitly. As mentioned earlier, the only instance in the literature of Theorem~\ref{thm:main} that we are aware of is the case of the UST of the complete graph $K_n$. In this case Grimmett~\cite{Gri:RandomTree} showed that the limiting object is an infinite path upon which we to each vertex an independent $\POISSON(1)$ branching process --- that is, a $\POISSON(1)$ branching process conditioned to survive forever. This is precisely $\kappa_W$ where corresponding graphon $W$ is just $W\equiv 1$.

The analogous continuity result for sparse graphs is also true, though in this case it is typically harder to describe explicitly the limiting object.

\begin{thm}[{\cite[Proposition~7.1]{MR2354165}}]\label{thm:sparse}
	Suppose that $\{G_n\}$ is a Benjamini--Schramm convergent sequence of connected graphs and let $T_n$ be a UST of $G_n$. Then there exists a random rooted tree $(T,\rho)$ such that $T_n$ Benjamini--Schramm converges to $(T,\rho)$.
\end{thm}
\noindent The limiting object in the above theorem $(T,\rho)$ is the \emph{wired uniform spanning forest} of the Benjamini--Schramm limit $(G,\rho)$ of the graphs $\{G_n\}$, see~\cite{LyPe:ProbabilityTrees}.

\medskip

One can also ask whether the normalized number of spanning trees is continuous with respect to taking graph limits. Given a graph $G$, let $t(G)$ be the number of spanning trees in $G$. In the bounded-degree model, Lyons~\cite[Theorem 3.2]{Lyons2005} proved that $n^{-1} \log t(G)$ is continuous in the Benjamini--Schramm topology. In the dense model, the natural normalization of $t(G)$ is $n^{-1} t(G)^{1/n}$. For example when $G=K_n$ then by Cayley's formula, $n^{-1} t(G)^{1/n}$ tends to $1$, and when $G$ is a typical Erd\H{o}s--R\'enyi random graph $\mathbb G(n,p)$ for $p\in(0,1)$ fixed, then $n^{-1} t(G)^{1/n}$ tends to $p$.
However, one cannot expect that with no further assumptions continuity of this parameter with respect to the cut-metric will hold. Indeed, let $G_n$ formed by a clique of order $n-\frac{n}{\sqrt{\log n}}$ and a path of length $\frac{n}{\sqrt{\log n}}$ attached to it at an arbitrary vertex. The sequence $\{G_n\}$ converges to the complete graphon $W\equiv 1$ but the number of spanning is substantially lower than for complete graphs; indeed, in this case it can easily by seen using Cayley's formula that $n^{-1} t(G_n)^{1/n}$ tends to $0$.

However, when we impose a minimum degree condition on the graphs, we can infer the asymptotic normalized number of spanning trees from the limit graphon.
\begin{thm}\label{thm:count}
	Let $\delta>0$. Suppose that $G_1,G_2,\ldots$ is a sequence of simple connected graphs that converge to a graphon $W$. Suppose that the order of $G_n$ is $n$ and the minimum degree is at least $\delta n$. Then the number of spanning trees satisfies
	$$\lim_{n\rightarrow \infty}\frac{\sqrt[n]{t(G_n)}}{n}=\exp\left(\int_x\log(\deg_W(x))\right)\;.$$\qed
\end{thm}
Theorem \ref{thm:count} follows almost immediately from a result due to Kostochka~\cite{Ko:SpanningTrees} which states that if $1<d_1\le d_2 \le \ldots \le d_n$ is the degree sequence of a simple connected graph $G$ then for some absolute constant $C>0$ we have
\begin{equation}\label{eq:Kostochka}
\frac{\prod_i d_i}{d_1^{(Cn\log d_1)/d_1}}\le t(G)\le \frac{\prod_i d_i}{n-1}\, .
\end{equation}
To see that~\eqref{eq:Kostochka} yields Theorem~\ref{thm:count}, it is enough to recall that the degree distribution of a limit graphon is inherited from degree distribution of graphs that converge to it (see Lemma~\ref{lem:graphondegrees}\eqref{en:alldeg}). In the case of regular graphs, Theorem \ref{thm:count} can also be derived from \cite{Alon1990, Frieze2000}. \newline

Lastly, let us mention a result of a similar flavor to ours in the context of percolation \cite{BBChR2010}.  There, the authors show that the critical percolation probability of a dense graph is ${1 \over \lambda_n}$ where $\lambda_n$ is the largest eigenvalue of the adjacency matrix. In particular it follows that if two dense graphs are close in the cut-metric, then this threshold is also close. They also describe the limiting local structure of bond percolation on dense graph in terms of a branching process on the limiting graphon. While there is some resemblance to the branching process of Theorem~\ref{thm:main}, they are quite different.

\subsection{About the proof of Theorem~\ref{thm:mainthm2}} Our proof combines, for the first time to the best of our knowledge, two seemingly unrelated mathematical areas, Kirchhoff's electric network theory and Szemer\'edi's regularity lemma-like graph partitioning techniques. These two are shown here to work seamlessly together.

Kirchhoff's theory of electric networks~\cite{Kirchhoff} allows to compute the probability that a given edge $e=xy$ is in a UST of a connected graph. This probability is precisely the \emph{effective electric resistance} between $x$ and $y$, when we consider the graph as an electric network and let current flow from $x$ to $y$, see Section~\ref{sec:effres} and~\eqref{eq:kirchh}. Since there is an edge connecting $x$ to $y$, this quantity is always a number in $[0,1]$. This is the starting point of our proof.

Next we use partition theory (Section~\ref{sec:decomposition}) to decompose our graph $G$ into a bounded number of dense expanders so that different expanders of the decomposition are connected by $\Landauo(n^2)$ edges. Heuristically, the UST of $G$ is close the union of independent USTs on each of these dense expanders.

Thus, it is natural to study electric theory on dense expanders. It is intuitive (and easy to prove) that if $x$ and $y$ are two vertices in graph, then the effective resistance between $x$ and $y$ is at least ${1 \over \deg(x)}+{1 \over \deg(y)}$ since at the most efficient scenario the electric current splits equally from $x$ to all its neighbors and arrives to $y$ equally from all of $y$'s neighbors. Of course this lower bound is not sharp --- the graph could be the disjoint union of two large stars around $x$ and $y$ and an edge connecting $x$ and $y$, in which case the resistance between $x$ and $y$ is $1$.

However, when the graph is a dense expander, one can use random walks estimates and employ the fascinating and classical connection between random walks and electric networks, to deduce a corresponding approximate upper bound, see Corollary~\ref{cor:effresinexpander}. The random walk estimate we prove (Lemma~\ref{lem:revisitexpander}) states that if one starts a random walk on a dense expander from some vertex that is not $x$ or $y$, then the probability that $x$ is visited before $y$ is ${(1+\Landauo(1))\deg(x) \over \deg(x) + \deg(y)}$.

It is now quite pleasant to observe that Rayleigh's monotonicity~\eqref{eq:Rayleigh}, which states that the electric resistance can only decrease by enlarging a network, shows that this upper bound on the resistance holds in each expander in the decomposition of $G$, and the matching lower bound holds for \emph{most} edges of $G$ since there are $\Landauo(n^2)$ edges between components, see Lemma~\ref{cor:effresinexpander}.

This explains why for most edges in $G$ the probability that they are exhibited in the UST is the sum of the degree reciprocals. An iterative argument is presented in Section~\ref{sec:fixedtrees}, employing the spatial Markov property of the UST (Proposition~\ref{prop:ustcond}), to control the probability of events such as $B_\T(v,r) \cong T$.
There are some delicate technicalities to overcome involving the ``outside'' effects of the decomposition. Once these are overcome, one reaches the discrete version $\FREQ(T;G)$ of the parameter $\FREQ(T;W)$ of Theorem~\ref{thm:mainthm2} and we show that this parameter approximates the desired probability (Lemma~\ref{lem:congTquenched}).
In Section~\ref{sec:proofmain} it is shown that $\FREQ(T;G)$ is close to $\FREQ(T;W)$, its continuous counterpart, if $G$ is close to $W$ in the cut-distance, concluding the proof.

\subsection{Organisation of the paper}
We tried to write the paper so that it can be read by probabilists and graph theorists alike. For this reason we recall even concepts relatively well known to one of the communities in a pedestrian manner. Also, at places we try to convey an idea of a proof even when this idea is standard, but in only one of the two communities.

In Section~\ref{sec:decomposition} we introduce a suitable decomposition of dense graphs, into so-called linear expanders.  In Section~\ref{sec:localust} we prove the discrete estimate approximating the probability of the event $\PROB(B_\T(X,r)=T)$ by the discrete parameter $\FREQ(T;G)$. In Section~\ref{sec:proofmain} we show that $\FREQ(T;G)$ and $\FREQ(T;W)$ are close whenever $G$ and $W$ are close in the cut-metric. Lastly, in Section~\ref{sec:degreeextremal} we derive Theorem~\ref{cor.degdist}.

\section{Decomposing dense graphs into linear expanders}\label{sec:decomposition}
\subsection{Dense expanders}\label{ssec:densexpanders}

Informally, a graph is a dense expander if whenever a vertex set and its complement are of linear size (in the order of the graph) then there are quadratically many edges between these two parts. So, a primal example of a dense graph that is \emph{not} an expander is a disjoint union of two cliques of order $\frac n2$ with a perfect matching connecting them.

We give our definition of expansion for {\bf loopless multigraphs}. That is, self-loops are not allowed, and two vertices may be connected by several edges. The quantity $e(A,B)$ counts all ordered pairs $ab$ that form an edge, $a\in A$ and $b\in B$, \emph{including multiplicities}. Note that each edge with both endvertices in $A\cap B$ contributes twice to $e(A,B)$. For a vertex $v$ and a vertex set $A$, we write $\deg(v,A):=e(\{v\},A)$. We write $\deg(v):=\deg(v,V)$, where $V$ is the vertex set of the (multi)graph. Note that with these conventions, we have $2e(G)=\sum_{v\in V} \deg(v)$, and $\sum_{a\in A}\deg(a,B)=e(A,B)=\sum_{b\in B}\deg(b,A)$, for each vertex sets $A$ and $B$.

\begin{defn}\label{def:expander}
	We say that a loopless multigraph $H$ is a {\bf $\gamma$-expander} if for each $U\subset V(H)$, we have $e(U,V(H)\setminus U)\ge \gamma |U|(v(H)-|U|)$.
\end{defn}

We will later use a simple observation.  Removing edges from an expander can obviously render its expansion properties. However, if one removes edges touching only one vertex while leaving the degree of the vertex high, the expansion properties are not damaged by too much as the following simple proposition states.

\begin{proposition}\label{prop:stillexpander} Let $G$ be a $\gamma$-expander loopless multigraph on $m$ vertices and let $v\in V(G)$ be a vertex. Assume that the maximal number of edges between any two vertices in $V(G)\setminus\{v\}$ is at most $\ell \in\bN$. Consider the graph $G'$ obtained from $G$ by erasing some set of edges emanating from $v$ of size at most $\ell m$ so that the degree of $v$ in $G'$ is at least $\gamma m$. Assume also that $8\ell^2 \gamma^{-2} \leq m$ . Then $G'$ is a ${\gamma \over 2}$-expander.
\end{proposition}
\begin{proof}
	We need to prove that for any $U \subset V(G)$,
	\begin{equation}\label{eq:toprove} e_{G'}(U, V(G) \setminus U) \geq {\gamma \over 2} |U|(m-|U|) \, .
	\end{equation}
	By symmetry it is enough to prove it when $|U|\leq m/2$. We proceed by considering two cases. In the first case we assume $|U| \geq 4 \gamma^{-1} \ell$. Then, since we erased at most $\ell m$ edges, we have
	$$ e_{G'}(U, V(G)\setminus U) \geq \gamma |U|(m-|U|) - \ell m \geq {\gamma \over 2} |U|(m-|U|) \, ,$$
	since in this case $\gamma |U|(m-|U|) \geq 2\ell m$.
	
	In the second case we assume that $|U| \leq 4 \gamma^{-1} \ell$. If $U=\{v\}$, then \eqref{eq:toprove} follows since the degree of $v$ in $G'$ is at least $\gamma m$. If $v \not \in U$, then the maximum number of edges that could have been erased from $e_G(U, V(G)\setminus U)$ is at most $\ell |U|$, hence
	$$ e_{G'}(U, V(G) \setminus U) \geq \gamma |U|(m-|U|) - 4 \ell^2 \gamma^{-1} \geq {\gamma \over 2} |U|(m-|U|) \, ,$$
	since $\gamma |U|(m-|U|) \geq \gamma m$ and $8\ell^2 \gamma^{-2} \leq m$. Lastly, if $v \in U$ and $|U|>1$, then
	$$ e_{G'}(U, V(G)\setminus U) \geq e_{G}(U \setminus \{v\}, V(G) \setminus U \cup \{v\}) - \ell |U| \geq {\gamma \over 2} |U|(m-|U|) \, ,$$
	by the same logic as above, concluding the proof. \end{proof}

\subsection{Expander decomposition of dense graphs}

The main result of this section, Theorem~\ref{prop:expanderdecomposition}, asserts that each graph that is close to a nondegenerate graphon can be decomposed into a bounded number expanders that are almost isolated from each other. In Definition~\ref{def:expanderdecomposition} below we describe the expander decomposition that we actually use. Let us recall that the need of such a  decomposition (rather than a single expander) stems from examples such as that of a disjoint union of two cliques of order $\frac n2$ with a perfect matching connecting them mentioned at the beginning of Section~\ref{ssec:densexpanders}. Passing to a limit we see that in the graphon perspective, the perfect matching vanishes and we are left with two components.
Therefore, we now introduce graphon counterparts to the notion of graph connectivity and components, and give their basic properties.
\begin{defn}\label{defn:irreducible}
	A graphon $W$ on a ground space $(\Omega,\mu)$ is {\bf disconnected} if either $W=0$ a.e. or there exists a subset $A\subset \Omega$ with $0<\mu(A)<1$ such that $W=0$ a.e. on $A\times A^c$; otherwise $W$ is {\bf connected}.	
\end{defn}

We shall require a result of Bollob{\'a}s, Janson and Riordan \cite[Lemma 5.17]{BolJanRio:PhaseTransition} which enables us to decompose a graphon into (at most) countably many connected components.

\begin{lem}\label{lem:BJR-decomposition}
	Let $W:\Omega\times\Omega\rightarrow [0,1]$ be a graphon. Then there exists a partition $\Omega=\bigcup_{i=0}^{N}\Omega_i$ into measurable subsets with $0\le N\le \infty$ such that $\mu(\Omega_i)>0$ for $i\ge 1$, the restriction of $W$ to $\Omega_i\times\Omega_i$ is connected for each $i\ge 1$, and $W=0$ a.e. on $(\Omega\times\Omega)\setminus\bigcup_{i=1}^{N}(\Omega_i\times\Omega_i)$.
\end{lem}

Bollob\'as, Borgs, Chayes and Riordan \cite[Lemma 7]{BBChR2010} showed that connectivity implies an apparently stronger statement.

\begin{lem}\label{lem:connectedness-quantitative}
	Let $W:\Omega^2\rightarrow [0,1]$ be a connected graphon, and let $0<\alpha<\tfrac12$ be given. There is some constant $\beta=\beta(W,\alpha)>0$ such that $\int_{A\times A^c}W\ge \beta$ for every measurable subset $A\subset \Omega$ with $\alpha\le \mu(A)\le \tfrac12$.
\end{lem}
We can now give our definition of expander decomposition.
\begin{defn}\label{def:expanderdecomposition}
	Suppose that $G$ is a loopless multigraph of order $n$. We say that $V(G)=V_0\sqcup V_1\sqcup\ldots\sqcup V_k$ is a {\bf $(\gamma,\eta,\epsilon)$-expander decomposition} if
	\begin{itemize}
		\item[\rm (G1)] $|V_0|\le \epsilon n$,
		\item[\rm (G2)] for each $i\in [k]$ we have that $e(V_i,V\setminus V_i)\le\eta |V_i|n$,
		\item[\rm (G3)] for each $i\in[k]$ and each $U\subset V_i$, we have $e(U,V_i\setminus U) \ge \gamma |U||V_i\setminus U|$.
	\end{itemize}
\end{defn}	
\begin{thm}\label{prop:expanderdecomposition}
	Suppose that $W:\Omega^2\rightarrow[0,1]$ is a nondegenerate graphon. Then for every $\epsilon,\eta>0$ there exist positive constants $\gamma=\gamma(W,\epsilon,\eta)$, $\xi=\xi(W,\epsilon,\eta)$ and $n_0=n_0(W,\epsilon,\eta)$ such that if $G$ is a graph with $v(G)>n_0$ and $\delta_\square(G,W)<\xi$ then $G$ admits a $(\gamma,\eta,\epsilon)$-expander decomposition.
	\setcounter{Prop26}{\value{thm}}
\end{thm}

For the proof of Theorem~\ref{prop:expanderdecomposition}, we shall need the following result. While we were not able to find an explicit reference, we consider this result folklore. To this end, we need the notion of subgraphons which we introduce now. Suppose that $W:\Omega^2\rightarrow[0,1]$ is a graphon on a probability space $(\mu, \Omega)$. Similarly to the graph case, for a set $\Lambda\subset \Omega$ we have the notion of a {\bf subgraphon of $W$ induced by $\Lambda$}. This is the restricted function $W[\Lambda]:=W_{\restriction\Lambda\times\Lambda}$. In order for $W[\Lambda]$ to be a graphon, we always have to consider it together with the renormalized probability space $(\frac{\mu(\cdot)}{\mu(\Lambda)}, \Lambda)$.
\begin{lem}\label{lem:severalcompo}
	Suppose that $W:\Omega^2\rightarrow [0,1]$ is a nondegenerate graphon. Suppose that we have a partition $\Omega=\Omega^*\sqcup\bigsqcup_{i=1}^k \Omega_i$ such that for each $i\in[k]$ we have that $W$ is zero almost everywhere on $\Omega_i\times (\Omega\setminus \Omega_i)$. Then for every $\lambda>0$ there exists a number $\xi>0$ so that we have the following. If $G$ is an $n$-vertex graph with $\delta_\square(G,W)<\xi$ then there exists a partition $V(G)=\bigsqcup_{i=0}^k V_i$ such that for each $i\in[k]$,
	\begin{enumerate}
		\item[(a)] $|V_0|=\mu(\Omega^*)n\pm \lambda n$, and $|V_i|=\mu(\Omega_i)n\pm \lambda n$,
		\item[(b)] $e_G(V_i,V(G)\setminus V_i)<\lambda n^2$,
		\item[(c)] $\delta_\square(G[V_i],W[\Omega_i])<\lambda$.
	\end{enumerate}
\end{lem}

\begin{proof}[Proof of Theorem~\ref{prop:expanderdecomposition}]
	Suppose that we are given a graphon $W:\Omega^2\rightarrow [0,1]$ and two parameters $\epsilon,\eta>0$.
	By Lemma~\ref{lem:BJR-decomposition}, there exists a partition $\Omega=\bigsqcup_{i=1}^K\Omega_i$ (with $K\le\infty$) into components of $W$. Let $k\in\mathbb N$ be such that $\mu\left(\bigcup_{i=0}^k\Omega_i\right)\ge 1-\frac\epsilon4$. Set $\Omega^*:=\bigcup_{i=k+1}^K\Omega_i$. Let $\alpha=\min \left\{ \frac{\epsilon}{6k}, \min\limits_{i\in[k]}\frac{\mu(\Omega_i)}{40}\right\}$. Lemma \ref{lem:connectedness-quantitative} shows the existence of a positive constant $\beta=\beta(W,\alpha)$ such that
	\begin{equation}\label{eq:graphon-connectedness}
	\iint\limits_{A\times (\Omega_i\setminus A)}W \ge \beta\ \text{for all $i\in [k]$ and all $A \subset \Omega_i$ with $\alpha \mu(\Omega_i) \le \mu(A)\le \mu(\Omega_i)/2$}.
	\end{equation}
	
	Let $\gamma,\xi$ and $\lambda$ satisfy
	\begin{equation}
	\label{eq:JaHl}
	\gamma=\min\left\{\frac{\mu(\Omega_i)\eta}{48},\frac{\mu(\Omega_i)\beta}{500}\right\} \quad \mbox{and} \quad 0<\xi\ll\lambda\ll \min\left\{\beta,\min_{i\in [k]}\mu(\Omega_i)\eta\right\}\;.
	\end{equation}
	Suppose that $G$ is a graph given at the input of the proposition.
	
	Let $V(G)=V'_0\sqcup V'_1\sqcup\ldots\sqcup V'_k$ be a partition satisfying properties of Lemma~\ref{lem:severalcompo} for the graphon $W$ and its partition $\Omega=\Omega^*\sqcup\bigsqcup_{i=1}^k\Omega_i$, together with input error parameter $\xi$ and output error parameter~$\lambda$.
	We will modify this partition to obtain a $(\gamma,\eta,\eps)$-expander decomposition of $G$.
	
	Lemma~\ref{lem:severalcompo}(a) gives that
	\begin{equation}\label{eq:frpole}
	|V'_0|\le \tfrac12 \epsilon n.
	\end{equation}
	
	For each $i\in[k]$, we perform the following cleaning procedure. Let $U:\Omega_i^2\rightarrow [0,1]$ be a graphon representation of $G[V_i']$ on the (renormalized) probability space $\Omega_i$ such that we have $\|W[\Omega_i]-U\|_\square <\lambda$. Such a representation exists by Lemma~\ref{lem:severalcompo}(c).
	
	Let $P_i^0:=V_i'$ and $Q_i^0:=\emptyset$.
	Now, for $j=1,2,3,\ldots$ we proceed as follows. If there exists at least one set $X_i^j\subset P_i^{j-1}$ of size at most $\frac35|P_i^{j-1}|$ with $e(X_i^j,P_i^{j-1}\setminus X_i^j)< \gamma |X_i^j|n$, then we take this set, and let $Q_i^j:=X_i^1\cup X_i^2\cup \ldots \cup X_i^j$, $P_i^j:=V_i'\setminus Q_i^j$, and proceed with $j+1$. If no set $X_i^j$ exists, then we set $j(i):=j-1$, $V_i:=P_i^{j(i)}$, and terminate. Since the sets $X_i^j$ ($j=1,2,\ldots, j(i)-1$) are nonempty, we will stop eventually.
	
	For every $i\in [k]$ and every $j\in [j(i)]$, we have
	\begin{align}
		\label{eq:fS1}
		e(Q_i^{j},P_i^j)=\sum_{\ell=1}^je(X^{\ell}_i,P_i^j) \le \sum_{\ell=1}^je(X^{\ell}_i,P^{\ell-1}_i\setminus X^{\ell}_i) \le \gamma n \sum_{\ell=1}^j |X^{\ell}_i|=\gamma |Q_i^j|n.
	\end{align}
	
	\begin{claim26}\label{cl:sx} For each $i\in[k]$,
		\begin{align}	
			\label{eq:fS2}
			|V_i'\setminus V_i|&\le 3\alpha n.
		\end{align}
	\end{claim26}
	\begin{claimproof}
		Suppose to the contrary that $|V_i'\setminus V_i|\ge 3\alpha n$. Let $j\in\{0,1,2,\ldots,j(i)\}$ be the largest index for which
		\begin{equation}\label{eq:bd}
		|P_i^{j}|\ge \tfrac14|V_i'|.
		\end{equation}
		Now, there are two cases to consider. Either $|P_i^{j+1}|<\frac14|V_i'|$ and then we have $|P_i^j|\le (\frac14+\frac35)|V_i'|$ by the way we chose the set $X_i^{j+1}$. Another case is that $j=j(i)$, that is, we terminated in the step $j$. Then, by our assumption, $|V_i'\setminus P_i^j|=|V_i'\setminus V_i|\ge 3\alpha n$. Put together,
		\begin{equation}\label{eq:wtf}
		|Q^j_i|=|V_i'\setminus P_i^{j}|\ge \min\left\{(1-\tfrac14-\tfrac35)|V_i'|,3\alpha n\right\}= 3\alpha n,
		\end{equation}
		as $|V'_i| \ge \tfrac12 \mu(\Omega_i)n$ (by Lemma~\ref{lem:severalcompo}(a)) and $\alpha \le \tfrac{1}{40}\mu(\Omega_i)$ (by~\eqref{eq:JaHl}).
		
		We learn from \eqref{eq:fS1} that
		\begin{align}\label{eq:fysins}
			e(Q_i^{j},P_i^{j})& \le \gamma |Q_i^j|n\leByRef{eq:bd} 4\gamma |Q_i^j||P_i^j|\cdot\frac{n}{|V_i'|}\le 5\gamma |Q_i^j||P_i^j|\cdot\frac{1}{\mu(\Omega_i)}\lByRef{eq:JaHl}\tfrac{\beta}{100} |Q_i^j||P_i^j|.
		\end{align}

		Let $\Lambda\subset \Omega_i$ represent the vertices of $P_i^j$. We have $\mu(\Lambda)\ge \frac{1}{4}\mu(\Omega_i)-\lambda\ge \frac{1}{5}\mu(\Omega_i)$, due to Lemma~\ref{lem:severalcompo}(a) and~\eqref{eq:bd}. Similarly,~\eqref{eq:wtf} gives $\mu(\Omega_i\setminus \Lambda)\ge \alpha\mu(\Omega_i)$. Thus, \eqref{eq:graphon-connectedness} applies. We have
		\begin{equation*}
			\int_{\Lambda}\int_{\Omega_i\setminus \Lambda} U\ge
			\int_{\Lambda}\int_{\Omega_i\setminus \Lambda} W-\|W[\Omega_i]-U\|_\square
			\geByRef{eq:graphon-connectedness}\beta-\lambda\;,
		\end{equation*}
		which contradicts~\eqref{eq:fysins}.
	\end{claimproof}
	
	We have defined the sets $V_1,V_2,\ldots, V_k$. Set $V_0:=V_0'\cup \bigcup_{i=1}^k (V_i'\setminus V_i)$. Let us now check that $V(G)=V_0\sqcup V_1\sqcup \ldots \sqcup V_k$ is indeed a desired expander decomposition.
	
	As for property (G1), we have
	\[
	|V_0|=|V_0'|+ \sum_{i=1}^{k}|V_i'\setminus V_i|\leBy{\eqref{eq:frpole}, Cl\ref{cl:sx}}\tfrac12 \epsilon n+3\alpha kn\le \epsilon n.
	\]
	
	For (G2), we first notice that
	\[
	|V_i|\geBy{\eqref{eq:fS2}} |V'_i|-3\alpha n \ge (\mu(\Omega_i)-\lambda)n-3\alpha n \ge \tfrac12 \mu(\Omega_i)n,
	\]
	as $\alpha \le \tfrac{1}{12}\mu(\Omega_i)$. Thus we find, as required,
	\begin{align*}
		e(V_i,V\setminus V_i) \le e(V'_i,V\setminus V'_i)+e(V'_i\setminus V_i,V_i) \leByRef{eq:fS1} \lambda n^2 +4\gamma|V'_i\setminus V_i|n \leBy{Cl\ref{cl:sx}} \lambda n^2+12\gamma \alpha n^2 \le \eta |V_i|n,
	\end{align*}
	where the last inequality holds since $|V_i|\ge \tfrac12\mu(\Omega_i)n$, $\lambda \ll \mu(\Omega_i)\eta$, and $\gamma \le \frac{\mu(\Omega_i)\eta}{48}$.
	
	Finally, property (G3) follows immediately from the stopping condition.
	This completes our proof of Theorem \ref{prop:expanderdecomposition}.
\end{proof}

\subsection{Properties of the expander decomposition}

For the proof of Theorem~\ref{thm:mainthm2} we argue that the majority of vertices do not ``see'' much beyond the component in the expander decomposition they belong to. This is formalized in the following definitions.
 
\begin{defn} \label{def:good}
	Suppose that $G$ is a loopless multigraph of order $n$. Assume that $V(G)=V_0\sqcup V_1\sqcup\ldots\sqcup V_k$ is some expander decomposition of $G$. For a vertex $v$ we write $i(v)$ for the unique $i\in\{0,1,\ldots,k\}$ such that $v \in V_i$. Given $\alpha>0$ and $\eps>0$ we say that $v\in V\setminus V_0$ is $(\alpha,\eps)$-{\bf good} with respect to the decomposition if the following hold:
	\begin{enumerate}[label=(\alph*)]
		\item\label{en:firstJohn} $\deg(v) \geq \LandauOmega(\eps n)$,
		\item\label{en:secondJohn} $\deg(v; V_{i(v)}) \geq (1-\LandauO(\eps^2)) \deg(v)$,
		\item\label{en:thirdJohn} $\displaystyle \sum_{u \in V_{i(v)}, u \sim v} \big ( {1 \over \deg(u; V_{i(v)})} - {1 \over \deg(u)} \big ) \leq \LandauO(\alpha^{1/2})$,
		\item\label{en:fourthJohn} $\displaystyle \sum_{u \in V_{i(v)}, u \sim v}  {1 \over \deg(u; V_{i(v)})} \leq \LandauO(\alpha^{-1/4})$.
	\end{enumerate}
\end{defn}

\begin{defn}\label{def:decomp}
	Suppose that $G$ is a loopless multigraph of order $n$. Given numbers $\beta, \alpha, \gamma>0$ and $\eps\in(0,\alpha)$, we say that $G$ has an $(\beta, \alpha,\gamma,\eps)$-{\bf good-decomposition} if
	\begin{enumerate}[label=(\arabic*)]
		\item\label{en:arab1} $G$ admits a $(\gamma, \eps^5, \eps^5)$-expander decomposition $V(G)=V_0\sqcup V_1\sqcup\ldots\sqcup V_k$, and
		\item\label{en:arab2} At least $(1-\LandauO(\alpha^{1/4}))n$ vertices of $G$ are $(\alpha,\eps)$-good.
        \item\label{e:arab3} At least $(1-\LandauO(\beta))n$ vertices of $G$ have degree at least $\LandauOmega(\alpha^{1/10} n)$.
	\end{enumerate}
\end{defn}

Next we refine Theorem~\ref{prop:expanderdecomposition}.

\begin{lem} \label{lem:furtherdecomp} For any $\beta>0$ and any nondegenerate graphon $W:\Omega^2\rightarrow[0,1]$, there exist $\alpha,\eps, \gamma,\xi>0$ with $\beta \gg \alpha \gg \eps \gg \gamma \gg \xi$ such that if $G$ is a simple graph on $n \geq \xi^{-1}$ vertices with $d_\square(G,W)\leq \xi$, then $G$ has a $(\beta, \alpha, \eps, \gamma)$-good-decomposition.
\end{lem}
\begin{proof}
	Let $\beta$ and $W$ be given. Since $W$ is nondegenerate, there exists  $\alpha>0$ such that any $m$-vertex graph ($m$ is arbitrary) that is $\xi_1$-close (for $\xi_1>0$ sufficiently small) to $W$ has at least $(1-\beta)m$ of degrees at least $\alpha^{1/10} m$, so requirement~\ref{e:arab3} of Definition~\ref{def:decomp} holds. Similarly, we can find constants $\epsilon \in (0,\alpha^{20})$ and $\xi_2>0$  such that any $m$-vertex graph ($m$ is arbitrary) which is $\xi_2$-close to $W$ has at most $\alpha m$ vertices of degrees at most $\epsilon m$.
	We apply Theorem~\ref{prop:expanderdecomposition}  with input $\eps^5$ and $\eta=\eps^5$ and retrieve $\gamma>0$ and $\xi_3>0$. We set $\xi:=\min(\xi_1,\xi_2,\xi_3)$. Suppose now that $G=(V,E)$ is a graph satisfying the assumptions of the lemma. Theorem~\ref{prop:expanderdecomposition} readily gives item~\ref{en:arab1} of Definition~\ref{def:decomp}.
	
	To show item~\ref{en:arab2}, we first note that by property (G2) of the expander decomposition we have that $e(V_i, V \setminus V_i) \leq \eps^5 n |V_i|$ for all $i\in [k]$. By summing over $i$ we deduce that
    \begin{equation}\label{eq:outedges} \sum_{i=1}^k e(V_i,V \setminus V_i) \leq \eps^5 n^2 \, .\end{equation}
	
	Denote by $S$ the set of vertices of $V\setminus V_0$ violating~\ref{en:secondJohn} of Definition~\ref{def:good} using~$1$ as the implicit constant in the term $\LandauO(\epsilon^2)$,\footnote{Note that later, in the proof of Lemma~\ref{lem:congTpurequenched}, we shall be forced to use larger implicit constants in~\ref{en:secondJohn} of Definition~\ref{def:good}.} that is,
	$$ S = \big \{ v \in V\setminus V_0 : \deg(v; V\setminus V_{i(v)}) \geq \eps^2 \deg(v) \big \} \, .$$
	Then by \eqref{eq:outedges} we have that
	$$ \eps^5 n^2 \geq \sum_{v \in S} \deg(v; V \setminus V_{i(u)}) \geq \eps^2 \sum_{v\in S} \deg(v) \, .$$
	Therefore $\sum _{v \in S} \deg(v) \leq \eps^3 n^2$, from which we learn that
	\begin{equation}\label{eq:boundS} |S \cap \{v : \deg(v) \geq \eps n \}| \leq  \eps^2 n \, .\end{equation}
	We deduce that
	\begin{equation}\label{eq:boundAB} \big | \big \{ v \in V \setminus V_0 : \deg(v) \leq \eps n \mathrm{\ or\ }\deg(v; V\setminus V_{i(v)}) \geq \eps^2 \deg(v) \big \} \big | \leq (\alpha + \eps^2)n \, .
	\end{equation}
	
	Next, for $i\in [k]$ we write
	$$ \sum_{v \in V_i} \sum_{u \in V_{i}, u \sim v} \big ( {1\over \deg(u; V_{i})} - {1 \over \deg(u)} \big ) = |V_{i}| - \sum_{u\in V_i}{\deg(u; V_i) \over \deg(u)} = \sum_{u \in V_i} {\deg(u; V\setminus V_i) \over \deg(u)} \, .$$
	We sum this over $i\in[k]$ and get that
	\begin{eqnarray*} \sum_{v \in V\setminus V_0} \sum_{u \in V_{i(v)}, u \sim v} \big ( {1\over \deg(u; V_{i})} - {1 \over \deg(u)} \big ) = \sum_{u \in V\setminus V_0} {\deg(u; V\setminus V_{i(u)}) \over \deg(u)} \leq (\alpha + \eps^2)n + \eps^2 n = \LandauO(\alpha n) \, ,\end{eqnarray*}
	where we bounded the ratio by $1$ for those vertices counted in~\eqref{eq:boundAB}, and by $\eps^2$ for the vertices that were not. From the last inequality we deduce that there cannot be more than $\LandauOmega(\alpha^{1/2} n)$ vertices $v\in V\setminus V_0$ such that requirement~\ref{en:thirdJohn} in the definition of $(\alpha,\eps)$-good is not satisfied.
	
	Lastly, to show~\ref{en:fourthJohn}, we have that
	$$ \sum_{i=1}^k \sum_{v\in V_i} \sum_{u\in V_i: u \sim v} {1 \over \deg(u; V_i)} = \sum_{i=1}^k |V_i| \leq n \, ,$$
	therefore there cannot be more than $\LandauOmega(\alpha^{1/4} n)$ vertices $v\in V\setminus V_0$ such that (d) is violated. This concludes our proof.
\end{proof}

Part~\ref{en:arab2} of Definition~\ref{def:decomp} asserts that there are many $(\alpha,\eps)$-good vertices in the graph, yet there could still be components of the decomposition $V_i$ in which the majority of their vertices are not $(\alpha,\eps)$-good. These cannot occupy too much of the mass. Indeed, for some $i\in [k]$, we say the set $V_i$ is $(\alpha,\eps)${\bf -big} if at least $(1-\LandauO(\alpha^{1/8}))|V_i|$ of its vertices are $(\alpha,\eps)$-good, and $e(G[V_i])\geq \LandauOmega(\alpha^{1/9}|V_i|n)$.

\begin{proposition}\label{prop:goodvis} Suppose that $G$ is a graph with $n$ vertices that has a $(\beta, \alpha,\gamma,\eps)$-good-decomposition (as in Definition~\ref{def:decomp}). Then
	$$ \sum_{i\in[k]: V_i \mathrm{\ is \ } (\alpha,\eps)\mathrm{-big}} |V_i| \geq (1-\LandauO(\beta^{1/8})) n \, . $$
\end{proposition}
\begin{proof}
Let $I_1$ be the indices $i\in[k]$ such that $V_i$ has $\LandauOmega(\alpha^{1/8} n)$ vertices that are not $(\alpha,\eps)$-good.  Since this is a $(\beta, \alpha,\gamma,\eps)$-good-decomposition we have that the total number of not $(\alpha,\eps)$-good vertices is $\LandauO(\alpha^{1/4} n)$. Hence,
$$ \sum_{i\in I_1} |V_i| \leq \LandauO(\alpha^{1/8} n) = \LandauO(\beta^{1/8} n) \, .$$
Next, let $I_2$ be the indices $i \in [k]$ such that $e(G[V_i]) = O(\alpha^{1/9}|V_i|n)$. Put $V' = \bigcup_{i\in I_2}V_i$. We have
$$ \sum_{v \in V'} \deg(v) \leq 2 \sum_{i\in I_2} e(G[V_i]) + \sum_{i \in I_2} e(V_i, V \setminus V_i) \leq O(\alpha^{1/9}|V'| n) + \eps^5 |V'|n = O(\alpha^{1/9}|V'| n) \, .$$

If $|V'| = \LandauOmega(\beta^{1/8} n)$, then by property~\ref{e:arab3} in Definition \ref{def:decomp}, we may bound $\sum_{v \in V'} \deg(v)$ from below by $\LandauOmega(\alpha^{1/10} |V'| n)$, giving a contradiction to the last estimate. The proof is concluded since $\sum_{i\in[k]} |V_i| \geq (1-\eps^5)n$.
\end{proof}

\section{Local neighborhoods of the UST via electric networks}\label{sec:localust}
\newcommand{\is}{\mathrm{\ is \ }}
\newcommand{\are}{\mathrm{\ are \ }}
\newcommand{\andd}{\mathrm{\ and \ }}

\subsection{Preliminaries}
Suppose that $G=(V,E)$ is a loopless multigraph. We denote by $\{X_t\}_{t \geq 0}$ the simple random walk starting from some (possibly random) vertex $X_0$. That is, $\{X_t\}_{t \geq 0}$ is a Markov chain with state space $V$ and transition matrix $p(x,y) = \frac{e(\{x\},\{y\})}{\deg(x)}$. We denote by $\PROB_v$ the probability measure of the simple random walk started at a vertex $v$. We will frequently use two stopping times: the hitting time of a vertex $v$ is the random variable $\tau_v:=\min\{t \geq 0:X(t)=u\}$ and the hitting time after zero of a vertex $v$ is the random variable $\tau_v^+:=\min\{t>0:X(t)=v\}$. Clearly when $X_0 \neq v$ these two random times are equal.

\subsubsection{Effective resistance.}\label{sec:effres} Our analysis relies on the relation between random walks, USTs and the theory of electrical networks. We briefly recall here the basic theory we will use and refer the reader to \cite[Chapter 2]{LyPe:ProbabilityTrees} for a comprehensive study. Given a graph $G=(V,E)$ we write $\overrightarrow{E}$ for the set of directed edges of size $2|E|$ which contain each edge of $E$ in both direction. Given two distinct vertices $u,v$ we say that an  antisymmetric function $f:\overrightarrow{E}\rightarrow \bR$ is a {\bf flow} from $u$ to $v$ if for each vertex $w\notin\{u,v\}$ the sum of $f$ over edges outgoing from $w$ is zero. A flow is called {\bf unit} if the sum of $f$ over edges outgoing from $u$ is~$1$. The {\bf effective resistance} $\RESISTANCE(u \lr v; G)$ between $u$ and $v$ is defined as the minimum energy $\mathcal{E}(f) = \sum_{e \in E} f(e)^2$ of any unit flow $f$ from $u$ to $v$. If $u$ and $v$ are not in the same connected component, we define effective resistance between them to be $\infty$. When it is clear what the underlying graph $G$ is we simply write $\RESISTANCE(u \lr v)$. From this definition it is immediate that if $G'$ is a subgraph of $G$, then
\begin{equation}
\label{eq:Rayleigh}
\RESISTANCE(u  \leftrightarrow v; G) \le \RESISTANCE(u \leftrightarrow v; G')\; .
\end{equation}
The latter inequality is also known as Rayleigh's monotonicity law. The \emph{discrete Dirichlet's principle} gives a dual definition of the effective resistance in terms of functions on the vertices. It states that
\begin{equation}\label{eq:dirichlet} \RESISTANCE(u\lr v) ^{-1} = \inf \left\{ \sum_{(x,y) \in E} (h(x) - h(y))^2 : h:V\to\mathbb{R}\, , \,  h(u)=0\, , \, h(v)=1 \right\} \, ,\end{equation}
see~\cite[Exercise 2.13]{LyPe:ProbabilityTrees}.
We will also a basic probabilistic interpretation of the effective resistance which can be found in \cite[Chapter 2]{LyPe:ProbabilityTrees}:
\begin{equation}\label{eq:revisitconductance}
\PROB_u\left[\tau_v<\tau_u^+\right] = \frac{1}{\deg(u) \RESISTANCE(u  \leftrightarrow v)} \, .
\end{equation}

\subsubsection{Uniform spanning trees} There is a fundamental connection between the uniform spanning tree and electric networks due to Kirchhoff \cite{Kirchhoff}. Let $G$ be connected loopless multigraph and $e=xy$ be an edge of the graph. As before we denote by $\T$ a UST of $G$. Kirchhoff's formula~\cite{Kirchhoff} (see also~\cite[Chapter 4]{LyPe:ProbabilityTrees}) states that for any edge $e=(x,y)$ of $G$ we have
\begin{equation}\label{eq:kirchh}
\PROB( e \in \T) = \RESISTANCE(x \leftrightarrow y) \, .
\end{equation}
Let $S$ be a subset of $E(G)$. We would like to condition on events of the form $S \subset \T$ or $S \cap \T = \emptyset$. We denote by $G/S$ the loopless multigraph obtained from $G$ by contracting the edges of $S$  and erasing any loops that has been formed, and by $G-S$ the graph $G$ with the edges of $S$ erased. The following is an easy and classical observation, see \cite[Chapter 4]{LyPe:ProbabilityTrees}.

\begin{proposition} \label{prop:ustcond} Let $G$ be a connected loopless multigraph and $S$ a subset of edges of $G$.
	\begin{enumerate}
		\item If $G-S$ is connected, then the UST $\T$ of $G$ conditioned on $S \cap \T = \emptyset$ has the distribution of the UST on $G - S$.
		\item If $S$ does not contain a cycle, then the UST $\T$ of $G$ conditioned on $S \subset \T$ has the distribution of $S \cup \T_{G/S}$ where $T_{G/S}$ is a UST of $G/S$.
	\end{enumerate}
\end{proposition}

%In what follows $\alpha>0$ is arbitrary, $G$ is a connected graph on $n$ vertices, $W$ is a nondegenerate graphon and $T$ is a fixed tree of height $r$ and size $\ell$. We apply Lemma \ref{lem:furtherdecomp} and extract the corresponding $\eps,\gamma$, so that there exists $\xi=\xi(W,\alpha,\gamma,\eps)>0$ such that if $d_\square(G,W)\leq \xi$, then $G$ admits a $(\gamma, \eps^5,\eps^5)$-expander decomposition $V(G)=V_0\sqcup V_1\sqcup\ldots\sqcup V_k$.
% Given this decomposition and $i\in[k]$ we recall that a vertex $v \in V_i$ is called $\gamma$-good if $\deg(v, V\setminus V_i) \leq \gamma^{\ell} |V_i|$.

\subsubsection{Mixing time} Let $G=(V,E)$ be a finite connected loopless multigraph and consider the lazy simple random walk on it, that is, the Markov chain on the vertex set $V$ with transition probability $p(x,y) = {e(\{x\},\{y\})\over 2\deg(x)}$ whenever $x \neq y$ and $p(x,x) = 1/2$ for any vertex $x$. Let $\pi$ be the stationary distribution $\pi(x) = \deg(x) / 2|E|$ and for each two disjoint subsets of vertices $A,B$ we write
$$ Q(A,B) = \sum_{x\in A, y\in B} \pi(x)p(x,y) = e(A,B)/4|E| \, .$$
The Cheeger constant $\Phi_*$ is defined as
$$ \Phi_* = \min_ {S : \pi(S) \leq {1\over2}} {Q(S,V\setminus S) \over \pi(S)} \, ,$$
where $\pi(S) = \sum_{x\in S} \pi(x)$. This ``bottleneck'' ratio is frequently used to control the spectral gap of the lazy random walk from which we may bound its mixing time. Let $1=\lambda_1 > \lambda_2 \geq \cdots \geq \lambda_n \geq 0$ be the eigenvalues of the transition matrix $p$ (we have $\lambda_1 > \lambda_2$ since $G$ is connected, and $\lambda_n \geq 0$ since the chain is lazy, see \cite{LPW:Mixing}). A result by Jerrum and Sinclair~\cite{SiJe:spectral-gap}, Lawler and Sokal~\cite{LaSo:spectral-gap} and Alon and Milman \cite{AlonMilman} states that
\begin{equation}\label{eq:cheeger} \Phi_*^2/2 \leq 1-\lambda_2 \leq 2\Phi_*.\end{equation}

Assume now that $G$ is a $\gamma$-expander as in Definition \ref{def:expander} and that the number of parallel edges between any two vertices %and the number of self-loops at any vertex
is at most $f \geq 1$. Then the degree of each vertex is at most $f n$, hence $\pi(S) \leq f|S|n/2|E|$ for any $S\subset V$. Similarly, for any $S \subset V$ we have that $\pi(V\setminus S) \leq fn|V\setminus S|/2|E|$, so if $\pi(V\setminus S) \geq 1/2$ we get that $|V\setminus S| \geq |E|/fn$. Since the minimum degree in a $\gamma$-expander is at least $\gamma(n-1)$ we get that if $\pi(V\setminus S) \geq 1/2$, then $|V\setminus S| \geq \gamma(n-1)/2f$. Putting all this together we get that if $G$ is a $\gamma$-expander and $n\geq 2$, then
$$ \Phi_* \geq {\gamma |S| |V\setminus S| \over 2|S|fn} \geq {\gamma^2 \over 8 f^2} \, ,$$
from which we get by \eqref{eq:cheeger} that
\begin{equation} \label{eq:spectral} 1-\lambda_2 \geq {c \gamma^4 \over f^4} \, , \end{equation}
where $c=1/128$. Recall that the {\bf total variation} distance $\|\mu-\nu\|_{\mathrm{TV}}$ between two probability measures $\mu$ and $\nu$ on the same probability space is defined to be $\sup_A |\mu(A)-\nu(A)|$, where the $\sup$ is ranging over all events $A$. For $\eps>0$ the $\eps$-mixing-time $T_{\mathrm{mix}}(\eps)$ of the chain is defined as
$$ T_{\mathrm{mix}}(\eps) = \min \big \{ t : \|p^t(x,\cdot) - \pi(\cdot)\|_{\mathrm{TV}} \leq \eps \textrm{ for all } x\in V \big  \} \, .$$
The mixing time and the spectral gap are related via the following statement, see \cite[Theorem~12.4]{LPW:Mixing},
$$ T_{\mathrm{mix}}(\eps) \leq {1 \over 1-\lambda_2} \Big [ {1 \over 2} \log{1 \over \min_{x \in V} \pi(x)} + \log(1/2\eps) \Big ] \, .$$

Since $f\geq 1$ bounds the maximal number of parallel edges between any two vertices, %and the number of self-loops at any vertex,
we get that $|E|\leq fn^2$ and hence $\pi(x) \geq \gamma(n-1)/2|E| \geq c\gamma /(fn)$ for some universal constant $c>0$. We deduce from this, the above bound on $T_{\mathrm{mix}}(\eps)$ and~\eqref{eq:spectral} that 
\begin{equation}\label{eq:mixrelax}  T_{\mathrm{mix}}(\eps) \leq C f^4 \gamma^{-4} \big [ \log n + \log f/\gamma + \log \eps^{-1} \big ] \, ,\end{equation}
for some universal constant $C>0$.

\subsection{Random walks on dense expanders}

\begin{lem}\label{lem:revisitexpander}
	Suppose that $H$ is a loopless multigraph on $n$ vertices that is a $\gamma$-expander and that the number of parallel edges between any two vertices is at most $f \geq 1$. Then for any two distinct nodes $u$ and $v$ and a node $w \neq v$ we have that $$\PROB_w\left[\tau_v<\tau_u^+\right]=\frac{\deg(v)}{\deg(u)+\deg(v)}+\LandauO(\gamma^{-7}f^7n^{-1}\log(n))\, .$$
\end{lem}
\begin{proof}
	Let us assume that there are no edges between $u$ and $v$ --- this can only matter for the assertion of the statement when $w = u$ and in this case affect the estimate by the probability that this edge is traversed on the first step of the random walk; the latter probability is at most $\LandauO(\gamma^{-1} f/n)$ the assumption, since $\deg(u) \geq \gamma n$.  This error is swallowed in the error estimate of the lemma.
	
	We denote by $\PROB^{\textrm{lazy}}$ the lazy random walk on the graph, that is, the random walk that with probability $1/2$ stays put and otherwise jumps to a uniformly chosen neighbor. It is clear that if $w$ is a vertex such that $w \not \in \{u,v\}$ then $\PROB^{\textrm{lazy}}_w(\tau_v<\tau_u^+) = \PROB_w(\tau_v<\tau_u^+)$. We will first show, via a coupling argument, that for any two vertices $w_1, w_2 \not \in \{u,v\}$ we have
	\begin{equation}\label{lem:revisitcoupling} \PROB^{\textrm{lazy}}_{w_1}(\tau_v < \tau_u^+) = \PROB^{\textrm{lazy}}_{w_2}(\tau_v < \tau_u^+) + \LandauO(\gamma^{-6} f^6 n^{-1} \log(n)) \, .\end{equation}
	
	Indeed, let $\{X_t\}_{t \geq 0}$ and $\{Y_t\}_{t \geq 0}$ be two lazy simple random walks starting at $w_1$ and $w_2$, respectively. We put $\eps=n^{-1}$ in \eqref{eq:mixrelax} and bound $\log(f/\gamma)$ by $f/\gamma$ and get that if $T=C\gamma^{-5}f^5\log(n)$, then $\|X_T - \pi\|_{\textrm{TV}} \leq n^{-1}$ and the same estimate holds for $Y_T$, hence $\|X_T - Y_T\|_{\textrm{TV}} \leq 2n^{-1}$ (where by $\|X_T - Y_t\|_{\textrm{TV}}$ we mean the total variation distance between the \emph{laws} of $X_T$ and $Y_T$).
	
By~\cite[Proposition 4.7]{LPW:Mixing} we deduce that we can couple the walks $\{X_t\}$ and $\{Y_t\}$ so that $X_T=Y_T$ with probability at least $1-2n^{-1}$. If this occurs we continue the coupling so that $X_t = Y_t$ for all $t \geq T$ by using the same random neighbor at each step of the walk. Thus, if both walks have not visited $u$ or $v$ between time $1$ and $T$, then the event $\{\tau_v < \tau_u^+\}$ occurs for $\{X_t\}$ if and only if it occurs for $\{Y_t\}$.  Since the minimal degree of $H$ is at least $\gamma (n-1)$ and the maximal number of parallel edges between any two vertices is at most $f$ we learn that the probability of visiting $u$ or $v$ between time $1$ and $T$ is at most $Tf/\gamma(n-1) \le 2C\gamma^{-6}f^6n^{-1}\log(n)$, concluding the proof of (\ref{lem:revisitcoupling}).
	
	We now continue the proof of the lemma for the case that $w=u$. If the walker starts at $u$ and $\tau_v < \tau_u^+$, then it cannot be lazy in the first step. Hence
	\begin{equation}\label{lem:revisitlazyvsnot} \PROB^{\textrm{lazy}}_u(\tau_v<\tau_u^+) = {1 \over 2} \PROB_u(\tau_v<\tau_u^+) \, .\end{equation}
	
	Consider now the Markov chain on the two states $\{u,v\}$ with transition probabilities
	$$ p(u,v) = \PROB^{\textrm{lazy}}_u(\tau_v < \tau_u^+) \, , \qquad p(v,u) = \PROB^{\textrm{lazy}}_v(\tau_u < \tau_v^+)\, ,$$
	with $p(u,u) = 1-p(u,v)$ and $p(v,v) = 1-p(v,u)$. This is the lazy random walk ``watched'' on the vertices $u$ and $v$. By summing over paths it is immediate that
	\begin{equation}\label{lem:revisitstationary} \deg(u) p(u,v) = \deg(v) p(v,u) \, .\end{equation}
	In order to visit $v$ before returning to $u$, the lazy walker must walk to a random neighbor in the first step, so
	\begin{equation}\label{lem:revisitfirsteq} p(u,v) = {1 \over 2} \PROB^{\textrm{lazy}}_{N(u)}(\tau_v < \tau_u^+) \, ,\end{equation}
	where $\PROB^{\textrm{lazy}}_{N(u)}$ indicates a uniform starting position from the set $N(u)$ of neighbors of $u$. Similarly, when starting from $v$, in order to return to $v$ before visiting $u$, the lazy walker can either stay put on the first step, or jump to a uniform neighbor of $v$ and from there visit $v$ before $u$, thus
	\begin{equation}\label{lem:revisitsecondeq} p(v,v) = {1 \over 2} + {1\over 2} \PROB^{\textrm{lazy}}_{N(v)}(\tau_v < \tau_u^+) \, .\end{equation}
	
	Since we assumed there are no edges between $u$ and $v$, by (\ref{lem:revisitcoupling}) we have that
	$$ \PROB^{\textrm{lazy}}_{N(v)}(\tau_v < \tau_u^+) = \PROB^{\textrm{lazy}}_{N(u)}(\tau_v < \tau_u^+) + \LandauO(\gamma^{-6}f^6n^{-1}\log(n)) \, .$$
	This together with (\ref{lem:revisitfirsteq}) and (\ref{lem:revisitsecondeq}) gives that $p(u,v)+p(v,u)=1/2 + \LandauO(\gamma^{-6}f^6n^{-1}\log(n))$. Together with (\ref{lem:revisitstationary}) and the fact that all degrees are at least $\gamma (n-1)$ and at most $fn$ gives that
	$$ p(u,v) = {\deg(v) \over 2(\deg(u) + \deg(v))} + \LandauO( \gamma^{-7}f^7 n^{-1}\log(n) ) \, ,$$
	concluding the proof of lemma when $w=u$ by (\ref{lem:revisitlazyvsnot}). The proof for any $w \not \in \{u,v\}$ can now be completed easily. By (\ref{lem:revisitcoupling}) and our assumption that there are no edges between $u$ and $v$ shows that
	$$ \PROB^{\textrm{lazy}}_{w}(\tau_v < \tau_u^+) = \PROB^{\textrm{lazy}}_{N(u)}(\tau_v < \tau_u^+) + \LandauO(\gamma^{-6}f^6n^{-1}\log n ) \, ,$$
	and so the lemma follows by~(\ref{lem:revisitfirsteq}).
\end{proof}

\begin{corollary}\label{cor:effresinexpander}
	Suppose that $H$ is a loopless multigraph on $n$ vertices that is a $\gamma$-expander and that the number of parallel edges between any two vertices is at most $f \geq 1$.
	Then for any two vertices $u \neq v$
	$$ \RESISTANCE(u \lr v) = (1 + \LandauO( \gamma^{-8} f^8 n^{-1} \log(n) ))\Big ( {1 \over \deg(u)} + {1 \over \deg(v)} \Big ) \, .$$
\end{corollary}
\begin{proof} Follows immediately by \eqref{eq:revisitconductance} and Lemma \ref{lem:revisitexpander} together with the fact that all degrees are at least $\gamma (n-1)$ and at most $fn$.
\end{proof}

We now extend Corollary~\ref{cor:effresinexpander} to the setting of a general dense graph.

\begin{lem}\label{lem:effresistinexpanderdecomp} Suppose that $G$ is a loopless multigraph with $n$ vertices given together with a $(\gamma,\eps^5,\eps^5)$-expander decomposition $V(G)=V_0\sqcup V_1\sqcup\ldots\sqcup V_k$. Assume further that the maximal number of parallel edges among any two pairs of vertices is at most $f\geq 1$. Assume that $\gamma^{-8} f^8 n^{-1}\log n \leq \eps$. Let $i \in [k]$ and $u \neq v$ be two distinct vertices of $V_i$. Then
	$$ (1- \LandauO(\eps)) \Big ( {1 \over \deg(u)} + {1 \over \deg(v)} \Big ) \leq \RESISTANCE(u \lr v) \leq (1- \LandauO(\eps)) \Big ( {1 \over \deg(u; V_i)} + {1 \over \deg(v; V_i)} \Big ) \, ,$$
	and if in addition $u$ and $v$ are $(\alpha,\eps)$-good, then
	$$ \RESISTANCE(u \lr v) =  (1 + \LandauO(\eps)) ) \Big ( {1 \over \deg(u)} + {1 \over \deg(v)} \Big ) \, .$$
\end{lem}
\begin{proof}
	Since $G[V_i]$ is a $\gamma$-expander on at least $\gamma n$ vertices, by Corollary \ref{cor:effresinexpander} together with Rayleigh's monotonicity~\eqref{eq:Rayleigh} we have
	\begin{eqnarray*} \RESISTANCE(u \lr v) &\leq& (1 + \LandauO( \gamma^{-8} f^8 n^{-1} \log(n) ) \Big ( {1 \over \deg(u; V_i)} + {1 \over \deg(v; V_i)} \Big ) \, ,
	\end{eqnarray*}
	giving the upper bound of the first assertion of the lemma. The upper bound of the second assertion immediately follows using the part~\ref{en:secondJohn} of Definition~\ref{def:good}.
	
	For the lower bound we will use Dirichlet's principle~\eqref{eq:dirichlet} and let $h:V(G) \to [0,1]$ be the function assigning $h(v)=1$, $h(u)=0$ and for any vertex $x \not \in \{u,v\}$ we put $h(x) = \deg(v)/(\deg(u) + \deg(v))$. By our assumption there are at most $f$ edges $(x,y)$ such that $x=u$ and $y=v$ in which $h(y)-h(x)=1$. Next, there are at most $\deg(u)$ edges $(x,y)$ for which $x=u$ and $h(y)-h(x) = \deg(v)/(\deg(u)+\deg(v))$. Similarly, there are at most $\deg(v)$ edges $(x,y)$ for which $y=v$ and $h(y)-h(x)=\deg(u)/(\deg(u)+\deg(v))$. All other edges $(x,y)$ of the graph have $h(x)-h(y)=0$. Hence by~\eqref{eq:dirichlet} we have
	$$ \RESISTANCE(u \lr v)^{-1} \leq f + {\deg(u) \deg(v)^2 + \deg(v) \deg(u)^2 \over (\deg(u) + \deg(v))^2} \leq {\deg(u)\deg(v) \over \deg(u) +\deg(v)} (1+ \gamma^{-4} f^2n^{-1}) \, ,$$
	where we used the fact that since each $V_i$ is a $\gamma$-expander, its cardinality must be $\LandauOmega(\gamma n)$ and hence $\deg(u)$ and $\deg(v)$ can be bounded below by $\LandauOmega(\gamma^2 n)$ and above by $fn$. This gives the required lower bound. \end{proof}

\subsection{The density of fixed trees in the UST}\label{sec:fixedtrees} The main result of this section, Lemma~\ref{lem:congTquenched}, allows express the frequency of a given fixed rooted tree $T$ in a graph using a discrete analogue of the parameter $\FREQ$, see Definition~\ref{def:freq} below.

Let us introduce some definitions and a setting that will be used throughout this section. In what follows we are always given an arbitrary $\beta>0$ and a nondegenerate graphon $W$. We then apply Lemma~\ref{lem:furtherdecomp} and extract the corresponding $\alpha,\eps,\gamma$ and $\xi$ so that if $G$ is a connected graph on $n$ vertices and $n$ is sufficiently large (as a function of $\alpha, \eps$ and $\gamma$) and $\delta_\square(G,W)\leq \xi$, then $G$ has a $(\beta, \alpha,\gamma,\eps)$-good decomposition as in Definition \ref{def:decomp}. We denote the given expander decomposition of $G$ by $V(G)=V_0\sqcup V_1\sqcup\ldots\sqcup V_k$. In light of the quantification of Lemma \ref{lem:furtherdecomp} we may assume that
$$ \beta \gg \alpha \gg \eps \gg \gamma \gg \xi \gg n^{-1} \, .$$

Next, let $T$ be a finite rooted tree of height $r$ and $\ell$ vertices denoted by $1, \ldots, \ell$ so that $1$ is the root of $T$. Given a graph $G$ we say that $\ell$ distinct vertices $(v_1, \ldots, v_\ell)$ of $G$ are {\bf compatible with $T$} if
the pairs
$$ T(v_1,\ldots,v_\ell) := \big \{ (v_q, v_t) : (q,t) \in E(T) \big \} \, ,$$
are all edges of $G$. Without loss of generality we may assume that the numbering $\{1,\ldots, \ell\}$ of the vertices of $T$ is such that there exists some $p\in \{2,\ldots, \ell\}$ such that the vertices of distance $r$ from the root (which all must be leaves) are $p,\ldots, \ell$.

\begin{defn}\label{def:ipure} Assume the setting as introduced at the beginning of Section~\ref{sec:fixedtrees}. Given some fixed $\ell \geq 1$ and $i\in[k]$ we say that an $\ell$-tuple of distinct vertices of $G$ are {\bf $i$-pure} if each vertex in the $\ell$-tuple belongs to $V_i$ and is $(\alpha,\eps)$-good.
\end{defn}

Next we define $\FREQ(T;G)$ which is the discrete analogue of $\FREQ(T;W)$. Note the notation $\FREQ(T;G)$ does not reflect the fact that this parameter depends on the partition, and not just on the graph $G$.

\begin{defn}\label{def:freq}
	Assume the setting as introduced at the beginning of Section~\ref{sec:fixedtrees}. For each $i\in [k]$ we define
	$$\FREQ(T;G, i):=|\stab_T|^{-1}\sum_{\substack{(v_1, \ldots, v_\ell) \\ i\textrm{-pure} \\  \textrm{compatible with } T}} |V_i|^{-1}\exp\left(-\sum_{j=1}^{p-1} b(v_j) \right) \frac{\sum_{j=p}^\ell \deg(v_j) }{\prod_{j=1}^\ell \deg(v_j)}\;,$$
	where
	\begin{equation}
	\label{eq:defbgraph}
	b(v) = \sum_{u \in V_{i(v)}, u \sim v} {1 \over \deg(u)}\;.
	\end{equation}
Note that~\eqref{eq:defbgraph} is a graph counterpart to the graphon quantity defined in~\eqref{eq:defbb}.
	
Finally, let
	\begin{equation}\label{eq:mobil}\FREQ(T;G):= \sum_{\substack{i\in[k]: V_i \mathrm{\ is \ } (\alpha,\eps)-\mathrm{big}}} {|V_i| \over n} \cdot \FREQ(T;G, i) \, .\end{equation}
\end{defn}

We denote by $\T$ a sample of the UST of $G$ and by $B_\T(v,r)$ the graph-distance ball in $\T$ of radius $r$ around $v\in V(G)$ and we think about it as a subset of edges of $\T$. We will begin our proof with estimating the probability that $B_\T(v,r)$ is manifested on an $i$-pure $\ell$-tuple of vertices (as in Definition~\ref{def:ipure}) that are compatible with $T$; later we will see that that all other manifestations of $B_\T(v,r)$ are negligible.
For an $i$-pure $\ell$-tuple that is compatible with $T$ we write $B_\T(v_1,r)\congpure T(v_1,\ldots,v_\ell)$ for the event
\begin{itemize}
	\item the edges $T(v_1, \ldots, v_\ell)$ are in $\T$, and,
	\item for each $1 \leq j \leq p-1$, the edges of emanating from $v_j$ that are not in $T(v_1,\ldots, v_\ell)$
	and have both endpoints in $V_i$ are not in $\T$.
\end{itemize}

\begin{lem}\label{lem:discretemain} Assume the setting as introduced at the beginning of Section~\ref{sec:fixedtrees}, and let $\T$ be a UST of $G$.	
	Let $T$ be a fixed rooted tree with $\ell\geq 2$ vertices $1, \ldots, \ell$ and height $r \geq 1$ (as usual $T$ is rooted at $1$). Assume that the vertices at height $r$ of $T$ are $\{p,\ldots, \ell\}$ for some $2 \leq p \leq \ell$. Then for any $i \in [k]$ and any $i$-pure $\ell$-tuple $(v_1, \ldots, v_\ell)$ that is compatible with $T$ we have that
	$$ \PROB\big ( B_\T(v_1,r)\congpure T(v_1,\ldots,v_\ell) \big ) = (1 + \LandauO(\ell \alpha^{1/4})) \exp\left(-\sum_{j=1}^{p-1} b(v_j) \right)\cdot \frac{\sum_{j=p}^\ell \deg(v_j) }{\prod_{j=1}^\ell \deg(v_j)} \, ,$$
	where $b(v)$ is defined in~\eqref{eq:defbgraph}.
\end{lem}
\begin{proof} Assume that $n$ is large enough as in Lemma \ref{lem:effresistinexpanderdecomp}.
	We first show by induction that the probability that $T(v_1,\ldots,v_\ell) \subset E(\T)$ is
	\begin{equation}\label{eq:edgesareinUST} (1 + \LandauO(\ell \eps)) {\sum_{j=1}^\ell \deg(v_j) \over \prod_{j=1}^\ell \deg(v_j)} \, .\end{equation}
	Indeed, when $T$ contains only one edge this statement follows immediately from Kirchoff's formula~\eqref{eq:kirchh} and Lemma \ref{lem:effresistinexpanderdecomp}. If $T$ has more than one edge, assume without loss of generality that $\ell$ is a leaf in $T$ at distance $r$ from the root and that $(q,\ell)$ is an edge of $T$. We then use the induction hypothesis on the tree $T\setminus \{\ell\}$, which has $\ell-1$ vertices. We condition on the event $T(v_1,\ldots,v_\ell) \setminus \{(v_q,v_\ell)\} \subset E(\T)$ and contract these $\ell-1$ vertices to a single vertex of degree $\deg(v_1) + \ldots + \deg(v_{\ell -1})$. We shall make use of Proposition~\ref{prop:ustcond} when working with the contracted graph. Since the contracted multigraph is still a $\gamma$-expander with at most $\ell$ parallel edges between any two vertices, and the vertices $(v_1,\ldots,v_\ell)$ are all $(\alpha,\eps)$-good, we may apply Lemma \ref{lem:effresistinexpanderdecomp} and Kirchoff's formula~\eqref{eq:kirchh} to get that the probability that the edge $(v_{q}, v_\ell)$ of $G$ is in $\T$ is
	$$ (1-\LandauO(\eps)) \Big ( {1 \over \deg(v_\ell)} + {1 \over \deg(v_1) + \ldots + \deg(v_{\ell -1})} \Big ) \, .$$
	By our induction hypothesis we get that the required probability is
	$$ (1-\LandauO((\ell-1)\eps))(1-\LandauO(\eps)) {\sum_{q=1}^{\ell-1} \deg(v_q) \over \prod_{q=1}^{\ell-1} \deg(v_q)} \Big ( {1 \over \deg(v_\ell)} + {1 \over \deg(v_1) + \ldots + \deg(v_{\ell -1})} \Big ) \, ,$$
	concluding the proof of~\eqref{eq:edgesareinUST}.
	
	\medskip
	
	We now condition on the event $T(v_1,\ldots,v_\ell) \subset E(\T)$ and turn to compute the probability that all other edges of $G$ emanating from $v_1, \ldots, v_{p-1}$ which have both endpoints in $V_i$ are not in $\T$. That is, the event that $E_{v_j} \cap \T = \emptyset$ for all $1 \leq j \leq p-1$, where $E_{v_j}$ are all the edges of $G$ between $v_j$ and $V_i\setminus \{v_1, \ldots v_\ell\}$.
	
	Let $j$ be an integer $1 \leq j \leq p-1$ and condition additionally on all the edges of $(E_{v_1} \cup \cdots\cup E_{v_{j-1}}) \cap \T = \emptyset$ (if $j=1$ there is no further conditioning). We will prove that under this conditioning the probability that $E_{v_j} \cap \T = \emptyset$ is
	\begin{equation}\label{eq:edgesnotinUST} (1+ \LandauO(\alpha^{1/4})) \exp\left(-b(v_j)\right) \cdot { \sum_{q=j+1}^\ell \deg(v_q) \over  \sum_{q=j}^\ell \deg(v_q) } \, .\end{equation}
	Thus, once~\eqref{eq:edgesnotinUST} is proved, by multiplying it over $j=1,\ldots,p-1$, we get that conditioned on $\O$ and on $E_G(T) \subset E(\T)$, the probability that all the required edges are not in $\T$ is
	$$ (1+ \LandauO(\ell \alpha^{1/4})) \exp\left(-\sum_{j=1}^{p-1} b(v_j)\right)\cdot  {\deg(v_p)+\ldots+\deg(v_\ell) \over \deg(v_1)+\ldots+\deg(v_\ell)} \, .$$
	We multiply~\eqref{eq:edgesareinUST} by this and get the required assertion of the lemma.
	
	We are left to prove \eqref{eq:edgesnotinUST}. Enumerate the neighbors of $v_j$ which are not in $\{v_1,\ldots,v_\ell\}$ by $u_1,\ldots,u_{\deg(v_j; V_i)}$. For each $1 \leq m \leq \deg(v_j; V_i)$ we condition on the first $m-1$ edges being closed, by Proposition \ref{prop:ustcond} we remove these edges from the graph, and after this removal the graph remains an ${\gamma \over 2}$-expander by Proposition \ref{prop:stillexpander}. Thus, by Lemma \ref{lem:effresistinexpanderdecomp} we have that in this conditioned graph the resistance on the $m$-th edge is
	\begin{equation}\label{eq:midcalc} (1+\LandauO(\eps)) \Big ( r_m + {1 \over \sum_{q=j}^\ell \deg(v_q) - (m-1) } \Big )\, ,\end{equation}
	where $r_m$ satisfies
	\begin{equation}\label{eq:rm} {1 \over \deg(u_m)} \leq r_m \leq {1 \over \deg(u_m; V_i)} \, .\end{equation}
	We do not necessarily know if $u_m$ is $(\alpha,\eps)$-good itself, and that is why it is not necessarily the case that the two bounds on $r_m$ are up to $(1+\LandauO(\eps))$ apart from each other. However, we will use the fact that $v_j$ is $(\alpha,\eps)$-good and use property~\ref{en:thirdJohn} of this definition.
	
	By~\eqref{eq:midcalc}, the probability that all other edges emanating from $v_j$ are closed, conditioned on this already occurring for $v_1,\ldots, v_{j-1}$, is
	$$ \prod _{m=1}^{\deg(v_j)} \Big [ 1- (1+\LandauO(\eps)) \big ( r_m + {1 \over \sum_{q=j}^\ell \deg(v_q) - (m-1) } \big ) \Big ] \, ,$$
	which equals
	\begin{equation}\label{eq:prod} \exp \big (-(1+\LandauO(\eps)) \sum_m r_m \big ) \cdot \exp \Big (-(1+\LandauO(\eps)) \sum_m {1 \over \sum_{q=j}^\ell \deg(v_q) - (m-1) }\Big ) \, .\end{equation}
	Since $v_j$ is $(\alpha,\eps)$-good, by property~\ref{en:thirdJohn} of the definition we learn that
	$$ \sum_{m} r_m = (1+\LandauO(\alpha^{1/2})) \sum_m {1 \over \deg(u_m)} =  (1+\LandauO(\alpha^{1/2})) b(v_j) \, .$$
	Property~\ref{en:fourthJohn} of the same definition asserts that $b(v_j) \leq \alpha^{-1/4}$, hence the first term in~\eqref{eq:prod} is
	$$\exp \big (-(1+\LandauO(\eps)) \sum_m r_m \big ) = (1+\LandauO(\alpha^{1/4})) e^{-b(v_j)} \, .$$
	
	To handle the second term of~\eqref{eq:prod} we note that
	$$  \sum_m {1 \over \sum_{q=j}^\ell \deg(v_q) - (m-1) } = \log \Big ( {\sum_{q=j} \deg(v_q) \over \sum_{q=j+1} \deg(v_q)  } \Big ) + \LandauO((\eps n)^{-1}) \, .$$
	Since $\{v_1, \ldots, v_\ell\}$ are $(\alpha,\eps)$-good we have that the ratio inside the logarithm is at most $1+\ell\eps^{-1}$ and so the second term of \eqref{eq:prod} equals
	$$ (1+\LandauO(\eps \log(\eps^{-1})))  { \sum_{q=j+1}^\ell \deg(v_q) \over  \sum_{q=j}^\ell \deg(v_q) } \, .$$
	We multiply these two terms of~\eqref{eq:prod} and get that~\eqref{eq:edgesnotinUST} holds. \end{proof}

Using Lemma~\ref{lem:discretemain} we may estimate the probability that $B_\T(X,r) \cong T$ ``purely'', that is, that $E(B_\T(X,r)) \cap E(G[V_i])$ is a rooted tree that is isomorphic to $T$. We denote this event by $B_\T(X,r) \congpure T$. Note that $E(B_\T(X,r)) \cap E(G[V_i])$ need not even be a tree, so the events $B_\T(X,r) \congpure T$ and $B_\T(X,r) \cong T$ can be very different.

\begin{corollary} \label{cor:congTpure} Assume the setting as introduced at the beginning of Section~\ref{sec:fixedtrees}, and let $\T$ be a UST of $G$. Fix $i\in [k]$ and let $X$ be a uniformly chosen random vertex of $V_i$. Then
	$$ \PROB(B_\T(X,r) \congpure T) = (1+\LandauO(\ell \alpha^{1/4})) \FREQ(T;G, i) \, ,$$
	where $\FREQ(T;G, i)$ is defined in Definition~\ref{def:freq}.
\end{corollary}
\begin{proof}
	This is immediate since the event $E(B_\T(X,r)) \cap E(G[V_i])\cong T$ is the union over all $i$-pure tuples $(v_1,\ldots,v_\ell)$ of the event $B_\T(v_1,r)\congpure T(v_1,\ldots,v_\ell)$. These events are not disjoint, since for each automorphism of $\tau$ of $T$ that fixes the root, we can permute $(v_1,\ldots,v_\ell)$ according to $\tau$ and get the identical event. However, up to this invariance, the events are disjoint (which explains the factor of $|\stab_T|^{-1}$ in the definition of $\FREQ(T;G,i)$) and so Lemma \ref{lem:discretemain} gives the proof.
\end{proof}

Corollary~\ref{cor:congTpure} is an annealed statement, that is, the probability space is the product of the UST probability measure and an independent uniform vertex $X$. A quenched statement follows by a second moment argument. We will first need the following assertion.

\begin{lem}\label{lem:awayfromO} Suppose that $G$ is a connected graph on $n$ vertices and $S \subset V(G)$. Let $\T$ be a spanning tree of $G$ and let $X$ be a uniformly chosen vertex of some set $A\subset V(G)$. Then for any fixed rooted tree $T$ of height $r$ we have that
	$$ \PROB\left( B_\T(X,r) \cong T \andd V(B_\T(X,r)) \cap S \neq \emptyset\right) \leq |T|^2|A|^{-1}|S| \, .$$
	Furthermore, if in addition $G$ has some decomposition $V(G)=V_0\sqcup V_1\sqcup\ldots\sqcup V_k$, then for each $i\in[k]$,
	$$ \PROB\left( B_\T(X,r) \congpure T \andd V(B_\T(X,r)) \cap V_i \cap S \neq \emptyset\right) \leq |T|^2|A|^{-1}|S| \, .$$
\end{lem}
\begin{proof}
	We prove only the first statement; the second follows by the same logic. Fix $v \in S$. For each vertex $q$ of $T$, if $\T$ contains an induced copy of $T$ such that $v$ takes the place of $q$, then the probability that $B_\T(X,r) = T$ such that $v$ takes the place of $q$ is at most $\deg_T(q)/|A|$ since once we choose which edge of $\T$ that touches $v$ corresponds to the edge of $T$ that touches $q$ towards the root of $T$, then the corresponding root in $\T$ is chosen and $X$ has to be chosen precisely to be that root. We bound this probability by $|T|/|A|$ and use the union bound over the vertices $q$ of $T$ and $v$ of $S$.
\end{proof}

We are now ready to prove the quenched version of Corollary~\ref{cor:congTpure}. Observe that without the assumption that $V_i$ is $(\alpha,\eps)$-big $\FREQ(T;G,i)$ can be very small or event $0$ (for instance, $V_i$ could have size less than $\eps^5 n$ and if all the vertices of $V_0$ are connected to each vertex of $V_i$, then $V_i$ has no $(\alpha,\eps)$-good vertices at all). Thus, we cannot hope to have the have the concentration required for the following quenched statement without this assumption.

\begin{lem}\label{lem:congTpurequenched}
	Assume the setting as introduced at the beginning of Section~\ref{sec:fixedtrees}.
	Let $\T$ be a UST of $G$. Fix $i\in [k]$ and assume that $V_i$ is $(\alpha,\eps)$-big. Let $X$ be a uniformly chosen vertex of $V_i$. Then with probability at least $1-\LandauO(\ell \alpha^{1/8})$ the random tree $\T$ is such that
	$$ \PROB\left(B_\T(X,r) \congpure T\right) = (1+\LandauO(\alpha^{1/16})) \FREQ(T;G, i) \, .$$
\end{lem}
\begin{proof} We write by $Z=Z(\T)$ the $\T$-measurable random variable counting the number of vertices $v$
	of $V_i$ such that $B_\T(v,r) \congpure T$. Corollary \ref{cor:congTpure} is equivalent to the assertion that
	\begin{equation}\label{eq:1stmom} \EXP Z = (1+\LandauO(\ell \alpha^{1/4})) \FREQ(T;G, i) \cdot |V_i| \, .\end{equation}
	The second moment of $Z$ is
	\begin{equation}\label{eq:2ndmom} \EXP Z^2 = \sum_{v, v' \in V_i} \PROB\left( B_\T(v,r) \congpure T \andd B_\T(v',r) \congpure T\right) \, .\end{equation}
	
	We will show that for any $v\in V_i$,
	\begin{equation}\label{eq:condcalc} \sum _{v' \in V_i} \PROB\left( B_\T(v', r) \congpure T \, \mid B_\T(v,r) \congpure T \right) \leq (1+\LandauO(\ell \alpha^{1/4}))\FREQ(T;G,i)|V_i| \, .
	\end{equation}
	If we have this, then since Corollary \ref{cor:congTpure} gives that
	$$ \sum_{v \in V_i} \PROB\left( B_\T(v,r) \congpure T \right) = (1+\LandauO(\ell \alpha^{1/4})) \FREQ(T;G,i) \cdot |V_i| \, ,$$
	by putting both in \eqref{eq:2ndmom} we get that
	$$ \EXP Z^2 = (1+\LandauO(\ell \alpha^{1/4})) \big [ \FREQ(T;G, i) \cdot |V_i| \big ]^2 \, .$$
	Hence
	$$ \mathrm{Var}(Z) = \LandauO(\ell \alpha^{1/4}) [ \FREQ(T;G, i) \cdot |V_i| \big ]^2 \, ,$$
	and so by Chebychev's inequality we learn that
	$$ \PROB \left( |Z - \EXP Z| \geq \alpha^{1/16} \FREQ(T;G, i) \cdot |V_i| \right) \leq \LandauO(\ell \alpha^{1/8}) \, ,$$
	concluding the proof.
	
	To prove~\eqref{eq:condcalc} we fix $v \in V_i$ and condition on the event $B_\T(v,r) \congpure T$ and on the vertices $(v_1,\ldots,v_\ell)$ which form $B_\T(v,r)$. By Proposition~\ref{prop:ustcond} we contract the edges of $\T$ in $B_\T(v,r)$ and erase the edges we conditioned that are not in $\T$. Denote the resulting multigraph by $G^*$ and by $v^*$ the vertex to which the vertices $v_1,\ldots, v_\ell$ have been identified. We consider the loopless multigraph $G^*$ with the partition
	$$ V(G^*) = V_0\sqcup V_1\sqcup \ldots \sqcup V_{i-1}\sqcup V_i^* \sqcup V_{i+1}\sqcup \ldots\sqcup V_k \, ,$$
	where $V_i^*$ is $V_i$ with the vertices $v_1, \ldots, v_\ell$ replaced by $v^*$. We now claim that this partition is a $(\beta, \alpha, \gamma/2, \eps)$-good-decomposition (as in Definition~\ref{def:decomp}). First, by Proposition~\ref{prop:stillexpander} the graph $G^*[V_i^*]$ is still a $\gamma/2$-expander and so the partition $V_0\sqcup V_1\sqcup \ldots \sqcup V_{i-1}\sqcup V_i^* \sqcup V_{i+1}\sqcup \ldots\sqcup V_k$ is a $(\gamma/2,\eps^5,\eps^5)$-expander decomposition of $G^*$. Next we verify that the number of good vertices has not dropped by too much. Indeed, from each vertex $u \in V_i \setminus \{v_1, \ldots, v_\ell\}$ at most $\ell$ edges touching it were erased and hence it is immediate to check in Definition~\ref{def:good} that, as long as $n$ is large enough (in terms of $\eps$ and $\alpha$), if $u$ was $(\alpha,\eps)$-good in $G$, then it is $(\alpha,\eps)$-good in $G^*$ --- the constants may have changed, but they are swallowed in the $\LandauO(\cdot)$ and $\LandauOmega(\cdot)$ notation of Definition~\ref{def:good}.
	
	Thus we may apply Corollary~\ref{cor:congTpure} and obtain that
	\begin{equation}\label{eq:condcalcmid} \sum_{v' \in V_i^*} \PROB \left( B_\T(v',r) \congpure T \mid G^* \right) = (1+\LandauO(\ell \alpha^{1/4})) \FREQ (T; G^*, i)\cdot |V_i^*|\, .\end{equation}
	
	We wish to bound the sum in~\eqref{eq:condcalc} by the above sum, however, there are two important differences between the sums. The first is the difference in the input to the $\FREQ$ function.
    \begin{claim} If $V_i$ is $(\alpha,\eps)$-big, then
	$$c\gamma^{\ell+1} e^{-\ell/\gamma} \leq \FREQ(T; G, i) \leq \ell \gamma^{-\ell-1}\, ,$$
	for some universal constant $c>0$.
	\end{claim}
	\begin{claimproof} Since $V_i$ is $(\alpha,\eps)$-big we have at least $(1-\LandauO(\alpha^{1/8}))|V_i|$ vertices which are $(\alpha,\eps)$-good and these must span at least $\LandauOmega(\alpha^{1/9} |V_i|n)$ edges, since the number of edges of $G[V_i]$ touching non $(\alpha,\eps)$-good vertices of $V_i$ is at most $\LandauO(\alpha^{1/8}|V_i|n)$ and we also have $e(G[V_i])=\LandauOmega(\alpha^{1/9}|V_i|n)$. Thus the average degree of this graph is $d=\LandauOmega( \alpha^{1/9} n)$, and, by greedily removing vertices of degree $d/4$ we can obtain a subgraph of $G[V_i]$ of minimal degree $\LandauOmega( \alpha^{1/9} n)=\LandauOmega(\gamma n)$ such that all of its vertices are $(\alpha,\eps)$-good. Thus it is immediate to find $\LandauOmega(\gamma^\ell n^\ell)$ copies of the tree $T$. We deduce that the number of $\ell$-tuples counted in $\FREQ(T; G, i)$ is at least $\LandauOmega(\gamma^\ell n^\ell)$ and at most $n^\ell$. Since each vertex degree in $V_i$ is at most $n$ and at least $\gamma n$ we learn that each $\ell$-tuple contributes to $\FREQ(T; G,i)$ at most $\ell \gamma^{-\ell-1}n^{-\ell}$ and at least $\gamma e^{-\ell/\gamma} n^{-\ell}$ and the claim follows.
	\end{claimproof}

 By this claim, since $V_i$ and $V_i^*$ differ by $\ell-1$ vertices, and $\ell$-tuples of $V_i$ which one of vertices is in $\{v_1, \ldots, v_\ell\}$ can contribute at most $\LandauO(\ell \gamma^{-\ell-1} n^{-1})$ to $\FREQ(T;G, i)$ we learn that
	\begin{equation}\label{eq:freqs} \FREQ(T; G, i) = (1+ \LandauO(\ell \gamma^{-\ell-1} n^{-1})) \FREQ(T; G^*, i) \, ,\end{equation}
	and $n$ can be taken large enough to that the error in the $\LandauO$-notation above is $\LandauO(\ell \alpha^{1/4})$. This handles the first difference.
	
	The second difference is that $B_\T(v',r)$ may be isomorphic to $T$ in $G$ by using some of the edges in $T(v_1,\ldots,v_\ell)$ that were contracted to $v^*$ in $G^*$. In this case, it does \emph{not} necessarily hold that $B_\T(v',r) \cong T$ in $G^*$, so it is possible that this contribution to~\eqref{eq:condcalc} is large and not counted for in~\eqref{eq:condcalcmid}. We will show that this is not the case.
	
	We bound the LHS of~\eqref{eq:condcalc} as follows. The terms corresponding to $v' \in \{v_1,\ldots, v_\ell\}$ we bound by $1$ and get a negligible contribution of $\ell$. For any other $v'$ we split the event $B_\T(v',r)\congpure T$ according to whether $V(B_\T(v',r)) \cap V(B_\T(v,r)) = \emptyset$. This intersection is empty if and only if $v^* \not \in V(B_\T(v',r))$ in the graph $G^*$, and if it is empty, then it holds that $B_\T(v',r)\congpure T$ in $G^*$. Thus the LHS of~\eqref{eq:condcalc} is at most
	\begin{eqnarray*} \ell + \sum _{v' \in V_i^*} \PROB\left( B_\T(v', r) \congpure T \mid G^*\right) + \sum _{v' \in V_i^*} \PROB\left( B_\T(v', r) \congpure T \andd v^* \in V(B_\T(v',r)) \mid G^*\right) \, .\end{eqnarray*}
	The first term in the above is negligible, the second we bound using \eqref{eq:condcalcmid} and \eqref{eq:freqs} by the RHS of \eqref{eq:condcalc}. The third term we bound using Lemma~\ref{lem:awayfromO} applied on the graph $G^*$ with $S=\{v^*\}$ and $A=V_i^*$, giving the bound
	$$ \sum_{v' \in V_i^*} \PROB \left(B_\T(v',r) \congpure T \andd v^*\in V(B_\T(v',r)) \,\, \mid \,\, G^* \right) \leq \ell^2 \, ,$$
	which is also negligible. This completes the proof of~\eqref{eq:condcalc} and concludes the proof of the lemma.
\end{proof}

We now turn to estimating the event $B_\T(X,r) \cong T$ rather than the ``pure'' version of this event. To that aim we define
\begin{equation}\label{eq:defo} \O = \big \{ (x,y) \in \T : i(x) \neq i(y) \mathrm{\ or \ } x\in V_0 \mathrm{\ or \ } y \in V_0 \big \} \, ,\end{equation}
that is, $\O$ is the set of edges of $\T$ that are between components or contained in $V_0$. We first assert that this set cannot be too large.

\begin{lem}\label{lem:fewedgesbetweencomponents} Assume the setting as introduced at the beginning of Section~\ref{sec:fixedtrees}. Let $\T$ be a UST of $G$ and $\O \subset \T$ be defined in~\eqref{eq:defo}. Then
	$$ \EXP |\O| = \LandauO(\alpha^{1/4} n) \, .$$
\end{lem}
\begin{proof}
	Let us assume that $n$ is large enough as in Lemma \ref{lem:effresistinexpanderdecomp}. By Lemma \ref{lem:effresistinexpanderdecomp} and Kirchoff's formula \eqref{eq:kirchh} we have that
	$$ \EXP |\T \cap E(G[V_i])| \geq (1-\LandauO(\eps)) {1 \over 2} \sum_{\substack{u \in V_i}}\sum_{\substack{v \in V_i : v\sim u }} \big ({1 \over \deg(u)} + {1 \over \deg(v)} \big ) = (1-\LandauO(\eps)) \sum_{u \in V_i} {\deg(u; V_i) \over \deg(u)} \, .$$
	Hence
	$$ \EXP \left|\T \cap \bigcup_{i=1}^k E(G[V_i])\right| \geq (1-\LandauO(\eps)) \sum_{u \in V\setminus V_0} {\deg(u;  V_{i(u)}) \over \deg(u)} \, .$$
	For any $u$ that is $(\alpha,\eps)$-good we have by part~\ref{en:secondJohn} of Definition~\ref{def:good} that $\deg(u; V_{i(u)})\geq (1-\eps^2) \deg(u)$. Since at least $(1-\LandauO(\alpha^{1/4}))n$ vertices are $(\alpha,\eps)$-good we deduce that
	$$ \EXP \left|\T \cap \bigcup_{i=1}^k E(G[V_i])\right| \geq (1-\LandauO(\alpha^{1/4}))n \, .$$
	The proof is now completed since $|\T|=n-1$ with probability $1$.
\end{proof}

We now reach our final destination.

\begin{lem}\label{lem:congTquenched}
	Assume the setting as introduced at the beginning of Section~\ref{sec:fixedtrees}. Let $\T$ be a UST of $G$ and let $X$ be a uniformly chosen vertex of $G$. Then with probability at least $1-\LandauO(\ell \alpha^{1/8})$ for the random tree $\T$ we have
	$$ \PROB\left(B_\T(X,r) \cong T\right) = (1+\LandauO(\ell^2 \alpha^{1/16})) \FREQ(T;G) + \LandauO(\ell^2 \beta^{1/8}) \, .$$
\end{lem}
\begin{proof} Let $\O$ be defined in~\eqref{eq:defo}. We apply Lemma~\ref{lem:congTpurequenched}, together with Lemma~\ref{lem:fewedgesbetweencomponents} and Markov's inequality, and get that with probability at least $1-\LandauO(\ell \alpha^{1/8})$ the random tree $\T$ satisfies
	\begin{enumerate}
		\item $|\O| \leq \alpha^{1/8} n$, and
		\item\label{itemdva} For each $i\in[k]$ for which $V_i$ is $(\alpha,\eps)$-big we have that
		$$\PROB\left(B_\T(X,r) \congpure T \right) = (1+\LandauO(\ell \alpha^{1/16})) \FREQ(T;G, i) \, ,$$
		where $X$ is a uniformly chosen vertex of $V_i$.
	\end{enumerate}
	Let $X$ be a uniformly chosen vertex of $G$. By Proposition~\ref{prop:goodvis} the probability that $X$ is in either $V_0$ or in some $V_i$ that is \emph{not} $(\alpha,\eps)$-big is at most $\LandauO(\beta^{1/8})$. Conditioned on $X \in V_i$ we have that $X$ is uniform random vertex of $V_i$. Hence we may sum item~\eqref{itemdva} above over these $i$'s and get that with probability at least $1-\LandauO(\ell \alpha^{1/8})$ the random tree $\T$ satisfies
	$$\PROB(B_\T(X,r) \congpure T) = (1+\LandauO(\ell \alpha^{1/16})) \FREQ(T;G) + \LandauO(\beta^{1/8}) \, ,$$
	by definition of $\FREQ(T;G)$.
	
	Assume $\T$ satisfies these. If $B_\T(X,r) \cong T$ but not $B_\T(X,r) \congpure T$, then we must have that $B_\T(X,r) \cong T$ and $B_\T(X,r) \cap V(\O) \neq \emptyset$. Similarly, if $B_\T(X,r) \congpure T$ but not $B_\T(X,r) \cong T$, then we similarly must have that $B_\T(X,r) \congpure T$ and $B_\T(X,r) \cap V(\O) \neq \emptyset$. Since $|\PROB(A) - \PROB(B)|\leq \PROB(A\setminus B) + \PROB(B\setminus A)$ we have that
	\begin{eqnarray*} \left| \PROB\left(B_\T(X,r) \cong T\right) - \PROB\left(B_\T(X,r) \congpure T\right) \right| = \LandauO(\ell^2 \beta^{1/8}) \, ,\end{eqnarray*}
	where the last inequality is due to Lemma~\ref{lem:awayfromO} with $A=V(G)$ and $S=V(\O)$.
\end{proof}

\section{Proof of Theorems ~\ref{thm:mainthm2} and \ref{thm:main}}\label{sec:proofmain}
\subsection{Deriving Theorem~\ref{thm:mainthm2} from Lemma~\ref{lem:congTquenched} via anatomies}
\newcommand{\centrum}{\mathrm{center}}
In this section, we deduce Theorem~\ref{thm:mainthm2} from Lemma~\ref{lem:congTquenched}. The main technical result of this section, Lemma~\ref{lem:freqContinuous}, says that the quantity $\FREQ$ is continuous in a certain robust sense. To prove Lemma~\ref{lem:freqContinuous}, we need to group the elements of $\Omega$ for a given graphon $W:\Omega^2\rightarrow[0,1]$ to groups with similar degrees. Actually, we need a recursive refinement of this, as follows. We call the above partition of $\Omega$ according to the degrees \emph{ anatomy of depth 1}. Now, having defined an anatomy of depth $d$, \emph{anatomy of depth $d+1$} is a decomposition of $\Omega$ into groups in which elements have approximately similar degrees into every individual cell of the anatomy of depth $d$. Let us make this precise.

Suppose that $h\in\bN$ and $S$ is a finite set. Let $\mathfrak C(h,S)$ be the partition of $[0,1]^S$ into Voronoi cells generated by points $p\in \{0,\frac1h,\ldots,\frac{h-1}h,1\}^S$ (we assign each boundary point to one arbitrary neighboring cell to break ties). Note that these cells form $(h+1)^{|S|}$ cubes of the form $\prod_{i\in S} \left[\max\{0,\frac{2r_i-1}{2h}\},\min\{1,\frac{2r_i+1}{2h}\}\right]$, for some $\big\{r_i\in\{0,\ldots,h\}\big\}_{i\in S}$.\footnote{Strictly speaking, when some but not all coordinates $r_i$ are $0$ or $h$, these are not cubes but rectangular prisms. This is however not important.} In the degenerate case $S=\emptyset$, we define $\mathfrak C(h,S):=\{\emptyset\}$. We call the points $p$ the {\bf centers of the cells}. When $C\in \mathfrak C(h,S)$ is a Voronoi cell with center $p=(p_i)_{i\in S}$, for $i\in S$ we write $center_i(C):=p_i$.

Let us write $\mathtt{b}$ for an (abstract) element. Below, the only purpose of $\mathtt{b}$ will be to refer to a coordinate that will have to do with the function $b_W(\cdot)$.

Now, suppose that $d\in\bN$ and $\mathbf{h}\in\bN^d$. For $t\in[d]$, we write $\mathbf{h}_t$ for the $t$-coordinate of $\mathbf{h}$, and $\mathbf{h}\llbracket t\rrbracket$ for the $t$-dimensional vector obtained from $\mathbf{h}$ by removing the last $d-t$ components. For $h\in\bN$, let $\mathfrak D_{h,0}:=\mathfrak C(h,\emptyset)$ and for $d\in\bN$ and $\mathbf{h}\in\bN^d$ let $\mathfrak D_{\mathbf{h}}:=\mathfrak C\left(\mathbf{h}_d,\mathfrak D_{\mathbf{h}\llbracket d-1\rrbracket}\sqcup\{\mathtt{b}\}\right)$. We have
\begin{equation}\label{eq:VorCellBounded}
\sum_{t=0}^d\left|\mathfrak D_{\mathbf{h\llbracket t\rrbracket}}\right|= \hbar(\mathbf{h})\;,
\end{equation}
for a suitable tower-function $\hbar(\cdot)$.

Suppose $W:\Omega^2\rightarrow[0,1]$ is a graphon and $h\in\bN$ is arbitrary. Let $\mathcal A_\emptyset:=\Omega$. We call $\{\mathcal A_\emptyset\}=\{\mathcal A_C\}_{C\in\mathfrak D_{h,0}}$ {\bf anatomy of $W$ of depth 0}. Suppose that $d\in\bN$, $\mathbf{h}\in\bN^d$ and that we already know the anatomy $\{\mathcal A_C\}_{C\in\mathfrak D_{\mathbf{h}\llbracket d-1\rrbracket}}$ of $W$ of depth $d-1$, which is a partition of $\Omega$. Now, for each $\omega\in\Omega$ we consider the $|\mathfrak D_{\mathbf{h}\llbracket d-1\rrbracket}|$-tuple of degrees
$$\mathbf{deg}_{d}(\omega):=\left(\deg_W(\omega,\mathcal A_C)\right)_{C\in \mathfrak D_{\mathbf{h}\llbracket d-1\rrbracket}}\;.$$
Then for each $F\in \mathfrak D_{\mathbf{h}}$ we define $\mathcal A_F:=(\mathbf{deg}_{d})^{(-1)}(F)\cap (\exp(-b_W(\cdot)))^{(-1)}(F)$. In words, each cell $\mathcal A_F\subset \Omega$ has the property that for each $C\in \mathfrak D_{\mathbf{h}\llbracket d-1\rrbracket}$ and for each $\omega\in\mathcal A_F$ we have that
\begin{equation}\label{eq:propertyanatomies}
\deg_W(\omega,\mathcal A_C)=center_C(F)\pm\tfrac{1}{2\mathbf{h}_d}
\quad\mbox{and}\quad
\exp(-b_W(\omega))=center_{\texttt{b}}(F)\pm\tfrac{1}{2\mathbf{h}_d}
\;.
\end{equation}
We call $\{\mathcal A_F\}_{F\in\mathfrak D_{\mathbf{h}}}$  the {\bf anatomy of $W$ of depth $d$}.
Obviously, $\{\mathcal A_F\}_{F\in\mathfrak D_{\mathbf{h}}}$ is a partition of $\Omega$.
Consider an arbitrary $\omega\in\mathcal A_F$. Summing~\eqref{eq:propertyanatomies} over all $C\in \mathfrak D_{\mathbf{h}\llbracket d-1\rrbracket}$ (for which~\eqref{eq:VorCellBounded} tells us that there are $\hbar(\mathbf{h}\llbracket d-1\rrbracket)$ summands), we get
\begin{equation}
\label{eq:deganatomy}
\deg_W(\omega)=\sum_{C\in \mathfrak D_{\mathbf{h}\llbracket d-1\rrbracket}}center_C(F)\; \pm \tfrac{\hbar(\mathbf{h}\llbracket d-1\rrbracket)}{2\mathbf{h}_d}\;.
\end{equation}
For this reason, we shall call the number $\sum_{C\in \mathfrak D_{\mathbf{h}\llbracket d-1\rrbracket}}center_C(F)$ the {\bf degree of $\mathcal A_F$} (in $W$), and denote it by $\deg^{\mathrm{anat}}_W(\mathcal A_F)$.

Suppose $U:\Omega^2\rightarrow[0,1]$ and $X\subset \Omega$. Suppose that $d\in\bN$, $\mathbf{h}\in\bN^d$, and $\kappa\ge 0$ are given. Let $\mathcal B_\emptyset:=\Omega$ and for each $t\in[d]$ let $\{\mathcal B_C\}_{C\in\mathfrak D_{\mathbf{h}\llbracket t\rrbracket}}$ be a partition of $\Omega$. Suppose that for each $t\in[d]$, each $F\in \mathfrak D_{\mathbf{h}\llbracket t\rrbracket}$, each $C\in \mathfrak D_{\mathbf{h}\llbracket t-1\rrbracket}$ and each $\omega\in \mathcal B_F\setminus X$ we have
\begin{equation}
\left|\deg_U(\omega,\mathcal B_C)-center_C(F)\right|\le \tfrac1{2\mathbf{h}_t}+\kappa
\quad\mbox{and}\quad
\left|\exp(-b_U(\omega))-center_{\texttt{b}}(F)\right|\le \tfrac{1}{2\mathbf{h}_d}+\kappa
\;.
\end{equation}
We then say that $\left\{\{\mathcal B_F\}_{F\in\mathfrak D_{\mathbf{h}\llbracket t\rrbracket}}\right\}_{t=0}^{d}$ are {\bf $\kappa$-approximate anatomies of $U$ up to depth $d$ with exceptional set $X$}.

Note that when we take $\kappa=0$ and $X=\emptyset$ then we recover the notion of anatomies.

\bigskip

Our next lemma says that two graphons that are close in the cut-distance have similar anatomies.
\begin{lem}\label{lem:anatomiescontNEW}
	Suppose that $W:\Omega^2\rightarrow [0,1]$ is a nondegerate graphon and $d\in\bN$, $\mathbf{h}\in\bN^d$ are arbitrary. Let $\left\{\{\mathcal A_F\}_{F\in\mathfrak D_{\mathbf{h}\llbracket t\rrbracket}}\right\}_{t=0}^{d}$ be the anatomies of $W$ up to depth $d$.
	For every $\kappa>0$ there exists a number $\delta>0$ such that the following holds for every graphon $U:\Omega^2\rightarrow [0,1]$ with $\|W-U\|_\square<\delta$. There exists a set $X\subset \Omega$ of measure at most $\kappa$ so that $\left\{\{\mathcal A_F\}_{F\in\mathfrak D_{\mathbf{h}\llbracket t\rrbracket}}\right\}_{t=0}^{d}$ are $\kappa$-approximate anatomies of $U$ up to depth $d$ with exceptional set $X$.
\end{lem}

\begin{proof}
	Suppose that we are given $W$, $d$, $\mathbf{h}$, and $\kappa$ as above.
	
	Since $W$ is nondegenerate, there exists $\beta\in(0,10^{-6})$ such that the measure of the set
	$$S:=\left\{\omega\in\Omega:\deg(\omega)<16\sqrt[4]{\beta}\right\}$$
	is less than $\frac{\kappa^2}{100}$.
	Let $\delta:=\min\{\frac{\beta\kappa^4}{400},\frac{\kappa^2}{4\hbar(\mathbf{h})^2}\}$. Suppose that $U$ is given  with $\|W-U\|_\square<\delta$.
	It is enough to prove that there exists a set $X_{\texttt{b}}\subset \Omega$ of measure at most $\frac{\kappa}{4}$ such that for each $t\in[d]$ and each $F\in \mathfrak D_{\mathbf{h}\llbracket t\rrbracket}$, we have for each $\omega\in F\setminus X_{\texttt{b}}$ that
	\begin{equation}\label{eq:anatomies0}
	\left|\exp(-b_U(\omega))-center_{\texttt{b}}(F)\right|\le \tfrac{1}{2\mathbf{h}_d}+\kappa\;,
	\end{equation}
	and that (with $t$ and $F$ as above) for each $C\in \mathfrak D_{\mathbf{h}\llbracket t-1\rrbracket}$ we have that all but at most $\frac{\kappa}{4\hbar(\mathbf{h})^2}$ measure of elements $\omega\in \mathcal A_{F}$ satisfy
	\begin{align}
		\label{eq:anatomies1}
		\deg_U(\omega,\mathcal A_C)\ge center_C(F)-{1 \over 2\mathbf{h}_t}- \kappa\;,
	\end{align}
	and all but at most $\frac{\kappa}{4\hbar(\mathbf{h})^2}$  measure of elements $\omega\in \mathcal A_{F}$ satisfy
	\begin{align}
		\label{eq:anatomies2}
		\deg_U(\omega,\mathcal A_C)\le center_C(F) +{1 \over 2\mathbf{h}_t}+ \kappa\;.
	\end{align}
	The lemma will then follow from~\eqref{eq:VorCellBounded} by taking $X$ to be the union of $X_{\texttt{b}}$ together with the exceptional sets from~\eqref{eq:anatomies1}, \eqref{eq:anatomies2} over all $t$, $C$, and $F$.
	
	\medskip
	
	The following claim clearly implies~\eqref{eq:anatomies0}.
	\begin{claim}\label{cl:bclose}
		There exists a set $X_{\texttt{b}}$ of measure at most $\frac{\kappa}{4}$ such that for all $\omega\in\Omega\setminus X_{\texttt{b}}$ we have $|b_W(\omega)-b_U(\omega)|<\kappa$.
	\end{claim}
	For the proof of Claim~\ref{cl:bclose}, we need Claims~\ref{cl:Vanoce1}--\ref{cl:Vanoce3} below.
	\begin{claim}\label{cl:Vanoce1}
		Suppose that $\Gamma:\Omega^2\rightarrow[0,1]$ is a graphon and that $A\subset\Omega$. Then
$$\int_{x\in\Omega}\int_{\omega\in A} \frac{\Gamma(x,\omega)}{\deg_\Gamma(\omega)}\le\mu(A).$$
	\end{claim}
	\begin{claimproof}
		By Fubini's Theorem, we have
		\begin{align*}
			\int_{x\in\Omega}\int_{\omega\in A} \frac{\Gamma(x,\omega)}{\deg_\Gamma(\omega)}&=\int_{\omega\in A}\frac1{\deg_\Gamma(\omega)}\int_{x\in\Omega} \Gamma(x,\omega)=\int_{\omega\in A}\frac1{\deg_\Gamma(\omega)}\deg_\Gamma(\omega,A)\le \mu(A)\;.
		\end{align*}
	\end{claimproof}
	
	\begin{claim}\label{cl:Vanoce2}
		Suppose $\Gamma_1,\Gamma_2:\Omega^2\rightarrow [0,1]$ are two graphons and that $f:\Omega\rightarrow [0,C]$ is an arbitrary function. Then
		$$\int_{x\in\Omega}\left|\int_{\omega\in\Omega}f(\omega)(\Gamma_1(x,\omega)-\Gamma_2(x,\omega)\right|\le 2C\cdot \|\Gamma_1-\Gamma_2\|_\square\;.$$
	\end{claim}
	\begin{claimproof}
By \cite[(8.20)]{Lovasz2012} we can alternatively express as $\|\Gamma_1-\Gamma_2\|_\square$
		\begin{align*}
			\|\Gamma_1-\Gamma_2\|_\square=&\sup_{F,G:\Omega\rightarrow[0,1]}\left|\int_x\int_y G(x)F(y)(\Gamma_1(x,y)-\Gamma_2(x,y))\right|=\\
			&=\frac1C\cdot\sup_{F:\Omega\rightarrow[0,C],G:\Omega\rightarrow[0,1]}\left|\int_x\int_y G(x)F(y)(\Gamma_1(x,y)-\Gamma_2(x,y))\right|\;.
			\\
			&\ge \frac1{2C}\cdot\sup_{F:\Omega\rightarrow[0,C],G:\Omega\rightarrow[-1,1]}\left|\int_x\int_y G(x)F(y)(\Gamma_1(x,y)-\Gamma_2(x,y))\right|\;.
		\end{align*}
	The claim follows by considering in the latter supremum functions $$F(\cdot) := f(\cdot) \quad\mbox{and}\quad G(\cdot) := \mathrm{sgn}\left(\int_y f(y) (\Gamma_1(\cdot,y)-\Gamma_2(\cdot ,y)) \right)\;.$$
	\end{claimproof}
	
	For the last auxiliary claim, we shall introduce an auxiliary notion. Below we shall work with not necessarily symmetric $L^1$-functions $K:\Omega^2\rightarrow[0,+\infty)$. Note that the notions of cut-norm and cut-distance extend to this setting.\footnote{Note that in this case, the cut-norm really has to be defined over all rectangles in~\eqref{eq:cutnormdef}, and not only over all squares.} The following claim is then obvious.
	\begin{claim}\label{cl:Vanoce3}
		Suppose that $\epsilon>0$ and $\Gamma,\Gamma':\Omega^2\rightarrow[0,1]$ are two $L^1$-functions such that  for each $x,y\in\Omega$, $\max(\Gamma(x,y),\Gamma'(x,y))<(1+\epsilon)\min(\Gamma(x,y),\Gamma'(x,y))$. Then $\|\Gamma-\Gamma'\|_\square\le \epsilon$.\qedClaim
	\end{claim}	
	\begin{claimproof}[Proof of Claim~\ref{cl:bclose}]
		Let $$D:=\left\{\omega\in\Omega\setminus S:\deg_W(\omega)\neq\left(1\pm \tfrac{\beta^{0.3}\kappa^2}4\right)\deg_U(\omega)\right\}\;.$$
		Since for each $\omega\in D$ we have $|\deg_W(\omega)-\deg_U(\omega)|\ge 3\beta^{0.55}\kappa^2$ we get from $\|W-U\|_\square<\frac{\beta\kappa^4}{400}$ that $\mu(D)\le \frac{\kappa^2}{100}$.
		
		Put $A:=S\cup D$. We have $\mu(A)\le\frac{\kappa^2}{50}$.
		
		We have
		\begin{equation}\label{eq:cohen}
		\int_{x\in\Omega} \left|b_W(x)-b_U(x)\right|\le \int_{x\in\Omega} \left|\int_{\omega\in\Omega\setminus A}\frac{W(x,\omega)}{\deg_W(\omega)}-\frac{U(x,\omega)}{\deg_U(\omega)}\right|+\int_{x\in\Omega} \left|\int_{\omega\in A}\frac{W(x,\omega)}{\deg_W(\omega)}-\frac{U(x,\omega)}{\deg_U(\omega)}\right|
		\end{equation}
		The second term on the right-hand side of~\eqref{eq:cohen} is bounded by Claim~\ref{cl:Vanoce1} by at most $2\mu(A)\le\kappa^2/50$. Now, let us consider the $L^1$-function $\tilde U$, defined by
		$$\tilde U(x,y):=\begin{cases}\frac{\deg_W(y)}{\deg_U(y)} \cdot U(x,y) \quad\mbox{if $y\not\in A$}\\
		U(x,y)\quad\mbox{if $y\in A$}
		\end{cases}\;.
		$$
		Observe that $U$ and $\tilde U$ satisfy the assumptions of Claim~\ref{cl:Vanoce3} with error parameter $\left(\frac{\beta^{0.3}\kappa^2}4\right)$.
		Then the first term of the right-hand side of~\eqref{eq:cohen} can be rewritten as
		\begin{equation}\label{eq:leonard}
		\int_{x\in\Omega}\left|\int_{\omega\in \Omega} \frac{\mathbold{1}_{\Omega\setminus A}(\omega)}{\deg_W(\omega)}\cdot\left(W(x,\omega)-\tilde U(x,\omega)\right)\right|\;.
		\end{equation}
		Observe that the function $\frac{\mathbold{1}_{\Omega\setminus A}(\cdot)}{\deg_W(\cdot)}$ is bounded by above by $\frac1{16\sqrt[4]{\beta}}$. Claim~\ref{cl:Vanoce2} tells us that~\eqref{eq:leonard} is at most $$\frac1{8\sqrt[4]{\beta}}\cdot\left(\|W-U\|_\square+\|U-\tilde U\|_\square\right)\leBy{C\ref{cl:Vanoce3}}\frac{1}{3200}\beta^{3/4}\kappa^2+\frac{\beta^{0.3}\kappa^2}{4\cdot 8\beta^{1/4}}<\frac{\kappa^2}{25}\;.$$
		Plugging this back into~\eqref{eq:cohen} we obtain that $\int_{x\in \Omega}|b_W(x)-b_U(x)|<\frac{\kappa^2}4$. The claim follows.
	\end{claimproof}
	
	It now remains to prove~\eqref{eq:anatomies1} and~\eqref{eq:anatomies2}. We will only prove~\eqref{eq:anatomies1} since the proof of~\eqref{eq:anatomies2} is verbatim.
	So, suppose that~\eqref{eq:anatomies1} fails, i.e., the set $X:=\{\omega\in \mathcal A_F:\deg_U(\omega,\mathcal A_C)> center_C(F) +{1 \over 2\mathbf{h}_t}+ \kappa\}$ satisfies $\mu(X)>\frac{\kappa}{2\hbar(\mathbf{h})}$. By the definition of $\mathcal A_F$ we have for every $\omega\in X$ that $\deg_W(\omega,\mathcal A_C)\le center_C(F) +{1 \over 2\mathbf{h}_t}$. Therefore,
	\begin{align*}
		\int_{X\times \mathcal A_F} W&\le (center_C(F)+{1 \over 2\mathbf{h}_t})\mu(X)\quad\mbox{but}\\
		\int_{X\times \mathcal A_F} U&> (center_C(F)+{1 \over 2\mathbf{h}_t}+\kappa)\mu(X)\ge (center_C(F)+{1 \over 2\mathbf{h}_t})\mu(X)+\frac{\kappa^2}{2\hbar(\mathbf{h})}\quad.\\
	\end{align*}
	This is a contradiction to the fact that $\|U-W\|_\square<\delta$.
\end{proof}

For the proof of Lemma~\ref{lem:freqContinuous} we need the following result.
\begin{proposition}\label{lem:avgb}
	Suppose that $W:\Omega^2\rightarrow[0,1]$ is a nondegenerate graphon. Then
$$\EXP_{x\in\Omega}[b_W(x)]=1.$$
\end{proposition}
\begin{proof}
	By Fubini's Theorem,
	\begin{align*}
		\EXP_{x\in\Omega}[b_W(x)]=\int_x\left(\int_y\frac{W(x,y)}{\deg(y)}\;\mathrm{d}y\right)\mathrm{d}x=\int_y\left(\int_x \frac{W(x,y)}{\deg(y)}\mathrm{d}x\right)\mathrm{d}y=1\;.
	\end{align*}
\end{proof}

Lemma~\ref{lem:freqContinuous} puts a relation between quantities $\FREQ(T;W)$ and $\FREQ^-(T;G,V_0,E_0)$, where $G$ is a graph, $V_0\subset V(G)$ and $E_0\subset E(G)$, and $\FREQ^-(T;G,V_0,E_0)$ is defined as
$$\FREQ^-(T;G,V_0,E_0):=|\stab_T|^{-1}\sum_{\substack{v_1, \ldots, v_\ell\in V(G) \setminus V_0\\ \forall ij\in E(T):v_iv_j\in E(G)\setminus E_0}} \frac{1}{v(G)}\exp\left (-\sum_{j=1}^{p-1} b_G(v_j) \right) \frac{\sum_{j=p}^\ell \deg_G(v_j) }{\prod_{j=1}^\ell \deg_G(v_j)}\;.$$

We are now ready to state Lemma~\ref{lem:freqContinuous}. It says that the parameter $\FREQ(T;\cdot)$ is continuous in a certain sense. While it would be possible to prove continuity of $\FREQ(T;\cdot)$ for nondegenerate graphons with respect to the cut-distance, here we need to put a relation between $\FREQ(T;\cdot)$ of a graphon and the parameter $\FREQ^-(T;\cdot)$ of a graph that is close to that graphon. Let us note that while our proof of continuity is long, the statement itself is natural. Indeed, the definition~\eqref{eq:frewTW} is an integration involving products of values of the graphon over the edges of $T$ (similar to the way that is used to define the density of $T$ in the graphon), degrees in the graphon and the function $b(\cdot)$ from~\eqref{eq:defbb} which is also defined using degrees in the graphon. As subgraph densities are continuous with respect to the cut-metric, and so is the degree sequence (c.f.\ Lemma~\ref{lem:graphondegrees}), it is actually plausible to have continuity of many graph(on) parameters obtained by combinations thereof.
\begin{lem}\label{lem:freqContinuous}
	Suppose that $W:\Omega^2\rightarrow[0,1]$ is a nondegenerate graphon and $T$ is a fixed tree. For every $\epsilon>0$ there exists $a>0$ such that for every $\chi>0$ there exists $n_0\in\bN$ and $\delta>0$ such that if $G$ is a graph and $V_0\subset V(G)$, $E_0\subset E(G)$ satisfy
	\begin{enumerate}[label=(\alph*)]
		\item $n>n_0$,
		\item $\delta_\square(W,G)<\delta$,
		\item\label{ens:Tuan} $|V_0|\le a n$, $|E_0|\le an^2$, and
		\item\label{ens:mindeg} for each $v\in V(G)\setminus V_0$ we have $\deg_G(v)\ge\chi n$,
	\end{enumerate}
	then we have
	\begin{equation}
	\label{eq:contribute}
	|\FREQ(T;W)-\FREQ^-(T;G,V_0,E_0)|<\epsilon\;.
	\end{equation}
\end{lem}
\begin{proof}
	Suppose that $W$, $T$, and $\epsilon$ are given.
	
	Let $L$ be the height of $T$. Suppose that the vertices of $T$ are $[\ell]$. Suppose that $1$ is the root of $T$, suppose that the height of each vertex $i$ is denoted $g_i$ and that the vertices of $T$ are enumerated so that we have $0=g_1<g_2\le g_3\le\ldots \le g_{p-1}<g_p=g_{p+1}=\ldots=g_{\ell}=L$.
	
	In the course of deriving Theorem~\ref{thm:main} from Theorem~\ref{thm:mainthm2} in Section~\ref{sec:Theorem1Theorem2} we prove that $\FREQ(T;W)\le 1$.\footnote{At this moment, we have not established Theorem~\ref{thm:main} nor Theorem~\ref{thm:mainthm2}. However, the fact $\FREQ(T;W)\le 1$ did not rely on the validity of either of these theorems, but rather followed from making the connection to the branching process $\kappa_W$.} In particular, the function $f_W:\Omega^\ell\rightarrow [0,+\infty)$,
	\begin{equation}
	\label{eq:definef}
	f_W(\omega_1,\ldots,\omega_\ell):= \exp \left(-\sum_{j=1}^{p-1} b_W(\omega_j) \right)\cdot \frac{\sum_{j=p}^\ell \deg_W(\omega_j) }{\prod_{j=1}^\ell \deg_W(\omega_j)} \cdot\prod_{\substack{(i,j) \in E(T)}} W(\omega_i, \omega_j)
	\end{equation}
	is integrable. Let $\delta_0>0$ be such that  for any set $A\subset \Omega^\ell$, $\mu^{\otimes\ell}(A)<\delta_0$ we have
	\begin{equation}\label{eq:Iwillgotobednow}
	\int_A f_W\le \frac{\epsilon}{10}\;.
	\end{equation}
	
	Since $W$ is nondegenerate, there exists a number $\Delta\in(0,\frac{\delta_0}{4\ell})$ such that the set $$\Omega_{\mathrm{small}}:=\{\omega\in\Omega:\deg_W(\omega)<2\Delta\}$$ has measure less than $\frac{\delta_0}{2\ell}$.
	
	Let $a:=\min\{\frac{\delta_0}{10\ell},\frac{\epsilon\Delta^\ell}{400\ell^\ell}\}$. Now, suppose that $\chi$ is given.

	Let $\mathbf{h}\in\bN^{L+1}$ and $d',d,\tau,\kappa>0$ satisfy
	\begin{equation}
	\label{eq:bfh}
	\min\left(\Delta,\epsilon,\chi,\delta_0\right)\gg d'\gg d\gg\tau\gg \frac{1}{\mathbf{h}_1}\gg \frac{1}{\mathbf{h}_2}\gg \ldots \gg  \frac{1}{\mathbf{h}_L}\gg \frac{1}{\mathbf{h}_{L+1}}\gg \kappa> 0\;.
	\end{equation}
	We note that the above dependencies are tower-type, i.e.,
	\begin{equation}\label{eq:towerdependency}
	\frac{1}{\mathbf{h}_{i}}\gg  \frac{\hbar(\mathbf{h}\llbracket i\rrbracket)}{\mathbf{h}_{i+1}}\;.
	\end{equation}
	
	Let $\left\{\{\mathcal A_{F}\}_{F\in\mathfrak D_{\mathbf{h}\llbracket t\rrbracket}}\right\}_{t=0}^{L+1}$ be the anatomies of $W$ up to depth $L+1$.
	Let $\delta$ be given by Lemma~\ref{lem:anatomiescontNEW} for input graphon $W$ and parameters $\mathbf{h}$ and $\kappa$.
	
	Suppose now that the graph $G$ is given.
	Take a graphon representation $U$ of $G$ such that $\|W-U\|_\square<\delta$.
	Lemma~\ref{lem:anatomiescontNEW} tells us that there exists a set $X\subset \Omega$ of measure at most $\kappa$ such that $\left\{\{\mathcal A_{F}\}_{F\in\mathfrak D_{\mathbf{h}\llbracket t\rrbracket}}\right\}_{t=0}^{L+1}$ are a $\kappa$-approximate anatomies of $U$ up to depth $L+1$ with exceptional set $X$.
	
	Given a vertex $v\in V(G)$, we write $\Omega_v\subset \Omega$ for the set representing $v$ in $U$. We write $\Lambda_{\mathrm{vert}}:=\bigcup_{v\in V_0}\Omega_v$, and $\Lambda_{\mathrm{edge}}=\bigcup_{uv\in E_0}\left(\Omega_u\times \Omega_v\cup \Omega_v\times \Omega_u\right)$. By assumption~\ref{ens:Tuan} we have
	\begin{equation}\label{eq:boundsLambda}
	\mu(\Lambda_{\mathrm{vert}})\le a \quad \mbox{and} \quad \mu^{\otimes 2}(\Lambda_{\mathrm{edge}})\le 2a\;.
	\end{equation}
	
	To express $\FREQ^-(T;G,V_0,E_0)$, we introduce a counterpart to~\eqref{eq:definef} that reflects~$U$,
	\begin{equation*}
		f_U(\omega_1,\ldots,\omega_\ell):= \exp \left(-\sum_{j=1}^{p-1} b_U(\omega_j) \right)\cdot \frac{\sum_{j=p}^\ell \deg_U(\omega_j) }{\prod_{j=1}^\ell \deg_U(\omega_j)} \cdot\prod_{\substack{(i,j) \in E(T)}} U(\omega_i, \omega_j)\;
	\end{equation*}
	and a version $f^-_U$, which takes into the account the ``deleted'' vertices $V_0$ and edges $E_0$,
	\begin{equation*}
		f^-_U(\omega):=
		\begin{cases}
			0 &\mbox{if $\exists$ $i\in[\ell]$: $\omega_i\in\bigcup_{v\in V_0}\Omega_v$ or $\exists$ $ij\in E(T)$: $(\omega_i,\omega_j)\in\bigcup_{vw\in E_0}\Omega_v\times \Omega_w$,}\\
			f_U(\omega)	&\mbox{otherwise.}
		\end{cases}
	\end{equation*}
	Note that the assumption~\ref{ens:mindeg} implies that
	\begin{equation}\label{eq:fUbounded}
	f^-_U\le\frac{\ell}{\chi^\ell} \quad\mbox{on the whole domain $\Omega^\ell$}\;.
	\end{equation}
	
	We say that a map $\pi:[\ell]\rightarrow \left\{\mathcal A_F\right\}_{t=1,\ldots,L+1; F\in\mathfrak D_{\mathbf{h}\llbracket t\rrbracket}}$ is a {\bf depth preserving anatomy assignment} if for each $i\in[\ell]$ we have $\pi(i)=\mathcal A_{F}$ for some $F\in \mathfrak D_{\mathbf{h}\llbracket L+1-g_i\rrbracket}$. This means that the root $1$ is assigned an anatomy of depth $L+1$ while vertices further from the root are assigned anatomies of smaller depths. Note that depth preserving anatomy assignment partition the set $\Omega^\ell$ it the sense that
	\begin{equation}\label{eq:anatomyassignmentsexhausts}
	\Omega^\ell=\bigsqcup_{\substack{\pi:[\ell]\rightarrow \left\{\mathcal A_F\right\}_{t=1,\ldots,L+1; F\in\mathfrak D_{\mathbf{h}\llbracket t\rrbracket}}\\\mathrm{depth~preserving~anatomy~assignment}}} \prod_{j=1}^{\ell}\pi(j)\;.
	\end{equation}
	In particular,
	\begin{align}\label{eq:MEASURE1}
		1&=\sum_{\substack{\pi:[\ell]\rightarrow \left\{\mathcal A_F\right\}_{t=1,\ldots,L+1; F\in\mathfrak D_{\mathbf{h}\llbracket t\rrbracket}}\\\mathrm{depth~preserving~anatomy~assignment}}} \prod_{j=1}^{\ell}\mu\left(\pi(j)\right)\;\mbox{, and}\\
		\label{eq:MEASURE2}
		1&=\sum_{\substack{\pi:[\ell]\rightarrow \left\{\mathcal A_F\right\}_{t=1,\ldots,L+1; F\in\mathfrak D_{\mathbf{h}\llbracket t\rrbracket}}\\\mathrm{depth~preserving~anatomy~assignment}}} \prod_{j=[\ell]\setminus \{j*\}}^{\ell}\mu\left(\pi(j)\right)\;\mbox{for each $j^*\in[\ell]$.}
	\end{align}
	
	We need to define three additional classes of depth preserving anatomy assignments.
	\begin{itemize}
		\item We say that a depth preserving anatomy assignment $\pi$ is {\bf singular} if there exists $i\in[\ell]$ such that a positive measure of the elements $\omega\in\pi(i)$ satisfy $\deg_W(\omega)\le \Delta$, or $b_W(\omega)>\Delta^{-1}$.
		\item We say that a depth preserving anatomy assignment $\pi$ is {\bf dense} if for every vertex $j\in [\ell]$ and every child $j^*$ of $j$ we have that $center_{\pi(j^*)}(\pi(j))>d\cdot \mu(\pi(j^*))$. If this fails for some pair $jj^*$, then we call $j^*$ a {\bf witness}. (Of course, several witnesses may exist.)
		\item We say that a depth preserving anatomy assignment $\pi$ is {\bf non-problematic} if it is dense, and it is not singular.
	\end{itemize}
	For a depth preserving anatomy assignment $\pi$, we write $\mathbf{\Omega}_\pi=\prod_{i=1}^\ell (\pi(i)\setminus X)$.
	
	The next claim establishes some basic properties of singular depth preserving anatomy assignment.
	\begin{claim}\label{cl:oneandall}
		Suppose that $\pi$ is a singular depth preserving anatomy assignment, and that $i\in\ell$ is a witness of singularity. Then we have that $\pi(i)\subset \Omega_{\mathrm{small}}$ or that  $b_W(\omega')>\frac{1}{2\Delta}$ for all $\omega'\in\pi(i)$.
	\end{claim}
	\begin{claimproof}
		The definition of singular depth preserving anatomy assignment says that there exists $\omega\in\pi(i)$ such that $\deg_W(\omega)\le \Delta$, or $b_W(\omega)>\Delta^{-1}$.
		
		Let us first deal with the former case. A double application of~\eqref{eq:deganatomy} gives that for each $\omega'\in\pi(i)$ we have
		$$\deg_W(\omega')\le \deg^{\mathrm{anat}}_W(\pi(i))+ \frac{\hbar(\mathbf{h}\llbracket L-1-g_i\rrbracket)}{2\mathbf{h}_{L-g_i}}\le
		\deg_W(\omega)+ \frac{2\hbar(\mathbf{h}\llbracket L-1-g_i\rrbracket)}{2\mathbf{h}_{L-g_i}}\lBy{\eqref{eq:bfh},\eqref{eq:towerdependency}} 2\Delta\;.$$
		Therefore, in this case, $\pi(i)\subset \Omega_{\mathrm{small}}$.
		
		We can deal with the latter case, but using~\eqref{eq:propertyanatomies} instead of~\eqref{eq:deganatomy}. We have
		\begin{equation*}\label{eq:oneandall2}
			b_W(\omega')>\Delta^{-1}-\frac{4}{2\mathbf{h}_{L-g_i}}\ge \frac{1}{2\Delta}\quad\mbox{for all $\omega'\in\pi(i)$}\;.
		\end{equation*}
	\end{claimproof}
	The next key claim says that the integrals of $f_W$ and $f^-_U$ over sets corresponding to non--problematic depth preserving anatomy assignment are almost the same.
	\begin{claim}\label{suminsidebox}
		Suppose that $\pi:[\ell]\rightarrow \left\{\mathcal A_{F}\right\}_{t=1,\ldots,L+1; F\in\mathfrak D_{\mathbf{h}\llbracket t\rrbracket}}$ is a non-problematic depth preserving anatomy assignment. Let
		\begin{align*}
			e^\pi_{\mathrm{vert}}&:=\sum_{i=1}^\ell \mu\left(\pi(i)\cap \Lambda_{\mathrm{vert}}\right)\cdot \prod_{k\in V(T)\setminus \{i\}}\mu(\pi(k))\quad\mbox{, and}\\
			e^\pi_{\mathrm{edge}}&:=\sum_{ij\in E(T)}\mu^{\otimes 2}\left(\Lambda_{\mathrm{edge}}\cap \bigcup_{ij\in E(T)} \pi(i)\times \pi(j)\right)\cdot \prod_{k\in V(T)\setminus \{i,j\}}\mu(\pi(k))\;.
		\end{align*}
		Then for the quantities
		\begin{align*}
			Q_1&:=\int\displaylimits_{\mathbold{\omega}\in \mathbf{\Omega}_\pi}
			f_W(\mathbold{\omega})\;\mbox{and}
			\\
			Q_2&:=\int\displaylimits_{\mathbold{\omega}\in \mathbf{\Omega}_\pi}
			f^-_U(\mathbold{\omega})
		\end{align*}
		we have $Q_1=Q_2\cdot (1\pm d)\pm \frac{(e^\pi_{\mathrm{vert}}+e^\pi_{\mathrm{edge}})\ell}{\Delta^\ell}$.
	\end{claim}
	\begin{claimproof}
		Suppose that $j\in V(T)$. If $j\neq 1$, then we write $\parent(j)$ for the parent of $j$ in $T$.
		Let us write $F_j$ for the tree obtained from $T$ by erasing the edge from $j$ to $\parent(j)$ and taking the component containing $j$. When $j=1$, we take $F_j:=T$. Also, for $j=1$ we define $center_{\pi(j)}\left(\pi(\parent(j))\right):=1$.
		
		For $j\in V(T)$, write
		$$R_j:=\prod_{k\in V(F_j)\setminus\{j\}} center_{\pi(k)}\left(\pi(\parent(k))\right)\;.$$
		Here, recall the convention that a product over the empty set is~$1$. This in particular applies when $j$ is a leaf and $j\neq 1$.
		
		We first want to show that the quantities $$\int\displaylimits_{(\omega_1, \ldots, \omega_\ell)\in \mathbf{\Omega}_\pi}\prod_{\substack{(i,j) \in E(T)}} W(\omega_i, \omega_j) \quad\text{and}\quad\int\displaylimits_{(\omega_1, \ldots, \omega_\ell)\in \mathbf{\Omega}_\pi} \prod_{\substack{(i,j) \in E(T)}} U(\omega_i, \omega_j)$$
		are very close. To this end, for $j\in V(T)$ we define $A_j(\omega_j)=B_j(\omega_j):=1$ if $j$ is a leaf (different than the root) and otherwise we define
		\begin{align}
			\begin{split}
				\label{eq:defAB}
				A_j(\omega_j)&:=\int\displaylimits_{\{\omega_k\in\pi(k)\setminus X\}_{k\in V(F_j)\setminus\{j\}}}\prod_{(a,b) \in E(F_j)} W(\omega_a, \omega_b)\;,\\
				B_j(\omega_j)&:=\int\displaylimits_{\{\omega_k\in\pi(k)\setminus X\}_{k\in V(F_j)\setminus\{j\}}}\prod_{(a,b) \in E(F_j)} U(\omega_a, \omega_b)\;.
			\end{split}
		\end{align}
		Inductively for $t=L+1,L,\ldots,0$, assume that for each vertex $j\in V(T)$ at height $g_j=t$,\footnote{Note that the first step of the induction is satisfied trivially, as the height of $T$ is $L$.} for each $\omega_j\in \pi(j)$ we have
		\begin{align}
			\label{eq:jezekA}
			A_j(\omega_j)=R_j\cdot \exp(\pm \tfrac{v(F_j)}{\mathbf{h}_1})
		\end{align}
		and for every $\omega_j\in \pi(j)\setminus X$ we have
		\begin{align}
			\label{eq:jezekB}
			B_j(\omega_j)=R_j\cdot \exp(\pm \tfrac{v(F_j)}{\mathbf{h}_1})\;.
		\end{align}
		Now, let suppose that the statement is true for all vertices at height $t+1$. Let $j$ be an arbitrary vertex at height $t$, and let $\omega_j\in\pi(j)$ be arbitrary. When $j$ is a leaf then~\eqref{eq:jezekA} and~\eqref{eq:jezekB} hold trivially. So, suppose that $j$ has some children $j_1,j_2,\ldots,j_q$. Since these children are at height $t+1$, applying the induction hypothesis, we disintegrate~\eqref{eq:defAB} with respect to these children, and get
		\begin{align*}
			A_j(\omega_j)&=\prod_{c=1}^q
			\left(\int_{\omega_{j_c}\in\pi(j_c)\setminus X}W(\omega_j,\omega_{j_c})\int_{\{\omega_k\in\pi(k)\setminus X\}_{k\in V(F_{j_c})\setminus\{j_c\}}}\prod_{(a,b) \in E(F_{j_c})} W(\omega_a, \omega_b)\right)\\
			&=
			\prod_{c=1}^q\left(\int_{\omega_{j_c}\in\pi(j_c) \setminus X}W(\omega_j,\omega_{j_c}) \cdot A_{j_c}(\omega_{j_c})
			\right)
			\eqByRef{eq:jezekA}\prod_{c=1}^q\left(\int_{\omega_{j_c}\in\pi(j_c)\setminus X}W(\omega_j,\omega_{j_c}) \cdot R_{j_c}\exp\left(\pm \tfrac{v(F_{j_c})}{\mathbf{h}_1}\right)
			\right)\\
			&=\prod_{c=1}^q \deg_W(\omega_j,\pi(j_c)\setminus X)\cdot\prod_{c=1}^q \left( R_{j_c}\exp(\pm \tfrac{v(F_{j_c})}{\mathbf{h}_1})\right)\\
			&=\prod_{c=1}^q (\deg_W(\omega_j,\pi(j_c))\pm \kappa)\cdot\prod_{c=1}^q \left( R_{j_c}\exp(\pm \tfrac{v(F_{j_c})}{\mathbf{h}_1})\right)
			\;.
		\end{align*}
		Now, each of the factors in the first product equals to $center_{\pi(j_c)}(\pi(j))\pm \frac{1}{\mathbf{h}_{L+1-g_j}}$ by using the property~\eqref{eq:propertyanatomies}. Since $\pi$ is a dense depth preserving anatomy assignment and since $\mu(\pi(j_c))$ is a positive number which depends only on $\mathbf{h}\llbracket L-g_j\rrbracket$, we have $$\deg_W(\omega_j,\pi(j_c))\pm \kappa=center_{\pi(j_c)}(\pi(j))\cdot(1\pm\sqrt{\kappa})\;.$$ Since $(1\pm\sqrt{\kappa})^q=\exp(\pm \tfrac{1}{\mathbf{h}_1})$, we just verified~\eqref{eq:jezekA} for $j$ and $\omega_j$.
		
		We can get exactly the same calculations for $B_j(\omega)$, where $\omega\in \pi(j)\setminus X$ and get
		\begin{align*}
			B_j(\omega_j)&=\prod_{c=1}^q (\deg_U(\omega_j,\pi(j_c))\pm \kappa)\cdot\prod_{c=1}^q \left(R_{j_c}\exp(\pm \tfrac{v(F_{j_c})}{\mathbf{h}_1}) \right)
			\;
		\end{align*}
		Now, the fact  $\omega_j$ is a non-exceptional element of the cell $\pi(j)$ of the $\kappa$-approximate anatomy $\{\mathcal A_F\}_{F\in\mathfrak D_{\mathbf{h}\llbracket L+1-j\rrbracket}}$ (for $U$) implies~\eqref{eq:jezekB}.
		
		\medskip
		Clearly, the term $\int_{\omega_1\in\pi(1)\setminus X} A_1(\omega_1)$ corresponds to the term $\int \prod_{\substack{(i,j) \in E(T)}} W(\omega_i, \omega_j)$ in the definition of $Q_1$. We shall now control all the remaining terms. The idea is that these all are almost constant since we are working inside one anatomy cell. Let us first deal with the terms in the definition of~$Q_1$.
		
		As $\pi(j)$ is non-singular for each $j\in [\ell]$, the denominator of $\frac{\sum_{j=p}^\ell \deg_{W}(\omega_j) }{\prod_{j=1}^\ell \deg_{W}(\omega_j)}$ is at least $\Delta^\ell$.
		Further, the fluctuations in the denominator are at most $\sum_{j=1}^{\ell}\frac{\hbar(\mathbf{h}\llbracket L-g_j\rrbracket)}{\mathbf{h}_{L+1-g_j}}\le \tau$ around $C_1:=\prod_{j=1}^\ell \deg^{\mathrm{anat}}_W(\pi(j))$ by~\eqref{eq:deganatomy}. Similarly, the nominator is bounded from below by $\Delta$, and the fluctuations in the nominator are at most $\sum_{j=p}^\ell \frac{\hbar(\mathbf{h}\llbracket L-g_j\rrbracket)}{\mathbf{h}_{L+1-g_j}}\le \tau$ around $C_2:=\sum_{j=p}^\ell\sum_{F\in\mathfrak D_{\mathbf{h}\llbracket L-g_j\rrbracket}}center_F(\pi(j))$. The fluctuations in the term $\exp \left(-\sum_{j=1}^{p-1} b_{W}(\omega_j) \right)$ are at most $\tau$ around $C_3:=\prod_{j=1}^{p-1} center_{\mathtt{b}}(\pi(j))$. Also, since $\pi(j)$ is non-singular for each $j\in [p-1]$, we have $C_3>\exp(-\frac{\ell}{\Delta})$. The two upper-/lower- bounds we have derived using from the non-singularity will be used below to transform an additive error of the type $value\pm error$ to a multiplicative one, $value\cdot (1\pm\frac{error}{lower~bound~on~the~value})$. Therefore,
		\begin{align}
			\nonumber
			Q_1&=\int\displaylimits_{(\omega_1, \ldots, \omega_\ell)\in \mathbf{\Omega}_\pi}
			\exp \left(-\sum_{j=1}^{p-1} b_{W}(\omega_j) \right)\cdot \frac{\sum_{j=p}^\ell \deg_{W}(\omega_j) }{\prod_{j=1}^\ell \deg_{W}(\omega_j)} \cdot
			\prod_{\substack{(i,j) \in E(T)}} W(\omega_i, \omega_j)\\
			\nonumber
			&= \int\displaylimits_{(\omega_1, \ldots, \omega_\ell)\in \mathbf{\Omega}_\pi}
			(C_3\pm \tau)\cdot\frac{C_2\pm \tau}{C_1\pm \tau}\cdot\prod_{\substack{(i,j) \in E(T)}} W(\omega_i, \omega_j)\\
			\nonumber
			&=C_3\cdot\frac{C_2}{C_1}\cdot \left(1\pm 6\tau \cdot\exp(\tfrac{\ell}{\Delta})\right)\cdot
			\int\displaylimits_{(\omega_1, \ldots, \omega_\ell)\in \mathbf{\Omega}_\pi}
			\prod_{\substack{(i,j) \in E(T)}} W(\omega_i, \omega_j)\\
			\nonumber
			&=C_3\cdot\frac{C_2}{C_1}\cdot \left(1\pm 6\tau \cdot\exp(\tfrac{\ell}{\Delta})\right)\cdot \int_{\omega_1\in\pi(1)\setminus X} A_1(\omega_1)\\
			\nonumber
			&\eqByRef{eq:jezekA}C_3\cdot\frac{C_2}{C_1}\cdot \left(1\pm 6\tau \cdot\exp(\tfrac{\ell}{\Delta})\right)\cdot \mu(\pi(1)\setminus X)\cdot R_1\cdot \exp(\pm \tfrac{\ell}{\mathbf{h}_1})\\
			\label{eq:DonaldTrump}
			&=
			C_3\cdot\frac{C_2}{C_1}\cdot \mu(\pi(1)\setminus X)\cdot R_1\cdot \left(1\pm\tfrac{d}{3}\right)
			\;.
		\end{align}
		We now attempt to show that this quantity is close to $Q_2$. First, let us introduce a modified quantity in which even integration over $V_0$ and $E_0$ is considered,
		$Q^+_2:=\int\displaylimits_{\mathbold{\omega}\in \mathbf{\Omega}_\pi}
		f_U(\mathbold{\omega})$.
		Now, we claim that we get exactly the same bound as in~\eqref{eq:DonaldTrump} even for $Q^+_2$. Let us explain that there are no traps in doing so. First, we can use $\int_{\omega_1\in\pi(1)\setminus X} B_1(\omega_1)$ to control $\int \prod_{\substack{(i,j) \in E(T)}} U(\omega_i, \omega_j)$, exactly in the same way as we used $\int_{\omega_1\in\pi(1)\setminus X} A_1(\omega_1)$ to control $\int \prod_{\substack{(i,j) \in E(T)}} W(\omega_i, \omega_j)$, because~\eqref{eq:jezekB} is a perfect counterpart to~\eqref{eq:jezekA}. The remaining quantities appearing in the definition of~$Q_2$ are again centered around the constants $C_1$, $C_2$, and $C_3$ as above, except the fluctuations can be bigger by $\kappa$ since $\left\{\{\mathcal A_{F}\}_{F\in\mathfrak D_{\mathbf{h}\llbracket t\rrbracket}}\right\}_{t=0}^{L+1}$ are only $\kappa$-approximate anatomies for $U$ (as opposed to exact anatomies for $W$). However, $\kappa$ is the smallest constant in~\eqref{eq:bfh} and thus this additional error can be easily swallowed by other error terms. Third, the upper- and lower- bounds based on non-nonsigularity that we used to transform an additive errors to multiplicative ones are still valid when the additional approximation term $\kappa$ is taken into account. Therefore, we conclude that $Q_2^+=
		C_3\cdot\frac{C_2}{C_1}\cdot \mu(\pi(1)\setminus X)\cdot R_1\cdot \left(1\pm\tfrac{d}{3}\right)$. This finishes the proof.
		
		Now, let us show that we have $|Q^+_2-Q_2|<\frac{(e^\pi_{\mathrm{vert}}+e^\pi_{\mathrm{edge}})\ell}{\Delta^\ell}$. Clearly, we have $Q^+_2\ge Q_2$. On the other hand, if for some $(\omega_1, \ldots, \omega_\ell)\in \mathbf{\Omega}_\pi$ we have that $f^-_U(\omega_1, \ldots, \omega_\ell)=0$ but $f_U(\omega_1, \ldots, \omega_\ell)>0$ then for some $i\in[\ell]$ we have $\omega_i\in \pi(i)\cap \Lambda_{\mathrm{vert}}$ or for some $ij\in E(T)$ we have $(\omega_i,\omega_i)\in\pi(i)\times \pi(j)\cap \Lambda_{\mathrm{edge}}$.
	\end{claimproof}
	
	\begin{claim}\label{cl:notdensenotsingular}Suppose that $\pi:[\ell]\rightarrow \left\{\mathcal A_{F}\right\}_{t=1,\ldots,L+1; F\in\mathfrak D_{\mathbf{h}\llbracket t\rrbracket}}$ is a depth preserving anatomy assignment that is not singular. Suppose that $j^*\in[\ell]$ is a witness that it is not dense either. Then
		\begin{equation}
		\label{eq:Skoda}
		\int_{\prod_{j=1}^{\ell} \pi(j)}f_W<\frac{d\ell}{\Delta^\ell}\cdot \prod_{j\in[\ell]} \mu(\pi(j))+\frac{\ell}{\mathbf{h}_{L+1-g_{j^*}}\cdot\Delta^\ell}\cdot \prod_{j\in[\ell]\setminus \{j^*\}} \mu(\pi(j))\;.
		\end{equation}
	\end{claim}
	\begin{claimproof}
		Let us bound all the terms in~\eqref{eq:definef}. More precisely, let $kj^*\in E(T)$ be any edge that witnesses that $\pi$ is not a dense depth preserving anatomy assignment. To bound $f_W$, we shall use that
		\[
		f_W(\omega_1,\ldots,\omega_\ell)\le \exp\left(-\sum_{j=1}^{p-1}b_W(\omega_j)\right)\cdot \frac{\sum_{j=p}^{\ell}\deg_W(\omega_j)}{\Pi_{j=1}^{\ell}\deg_W(\omega_j)}\cdot W(\omega_k,\omega_{j^{*}}).
		\]
		The term $\exp \left(-\sum_{j=1}^{p-1} b_W(\omega_j) \right)$ is trivially at most~$1$. The nominator in the term $\frac{\sum_{j=p}^\ell \deg_W(\omega_j) }{\prod_{j=1}^\ell \deg_W(\omega_j)}$ is at most $\ell$ while the denominator is at least $\Delta^\ell$ by non-singularity. Finally, because $kj^*\in E(T)$ is an edge that witnesses that $\pi$ is not a dense depth preserving anatomy assignment. Then for every choice of $\omega_k\in \pi(k)$, we have that $\int_{\omega_{j^*}\in\pi(j^*)} W(\omega_{k},\omega_{j^*})=\deg_W(\omega_k,\pi(\omega_{j*}))\le d\cdot \mu(\pi(j^*))+\tfrac{1}{2\mathbf{h}_{L+1-g_{j^*}}}$. The claim now follows by integrating over the remaining dimensions.
	\end{claimproof}
	
	We are now ready to express $\FREQ(T;W)$ and $\FREQ^-(T;G,V_0,E_0)$ in a way that most terms will approximately cancel. To express $\FREQ(T;W)$, we work with the function~$f_W$ defined in~\eqref{eq:definef}. Also, we partition the integration over $\Omega^\ell$ in~\eqref{eq:frewTW} into terms corresponding to individual depth preserving anatomy assignments as in~\eqref{eq:anatomyassignmentsexhausts}.
	\begin{align}
		\begin{split}\label{eq:expressW}
			|\stab_T|\cdot \FREQ(T;W)&=\sum_{\pi \mathrm{~non-problematic}} \int_{\mathbf{\Omega}_\pi}f_W\\
			&~~+
			\sum_{\pi \mathrm{~non-problematic}}\int_{\prod_{j=1}^\ell \pi(j)\setminus\mathbf{\Omega}_\pi}f_W\\
			&~~ +\sum_{\pi \mathrm{~not~dense~and~not~singular}} \int_{\prod_{j=1}^\ell \pi(j)}f_W\\
			&~~+\sum_{\pi \mathrm{~singular}} \int_{\prod_{j=1}^\ell \pi(j)}f_W\;.
		\end{split}
	\end{align}
	
	Now, we have
	\begin{align}
		\begin{split}\label{eq:expressG}
			|\stab_T|\cdot \FREQ^-(T;G,V_0,E_0)&=\sum_{\pi \mathrm{~non-problematic}} \int_{\mathbf{\Omega}_\pi}f^-_U\\
			&+
			\sum_{\pi \mathrm{~non-problematic}}\int_{\prod_{j=1}^\ell \pi(j)\setminus\mathbf{\Omega}_\pi}f^-_U\\
			&~~ +\sum_{\pi \mathrm{~not~dense~and~not~singular}} \int_{\prod_{j=1}^\ell \pi(j)}f^-_U\\
			&~~+\sum_{\pi \mathrm{~singular}} \int_{\prod_{j=1}^\ell \pi(j)}f^-_U\;.
		\end{split}
	\end{align}
	We will show that the first sums in~\eqref{eq:expressW} and~\eqref{eq:expressG} cancel almost perfectly, and that all the remaining terms are negligible.
	
	Claim~\ref{suminsidebox} tells us that for the first sums in~\eqref{eq:expressW} and~\eqref{eq:expressG} we have
	\begin{equation}\label{eq:123}
	\sum_{\pi \mathrm{~non-probl.}} \int_{\mathbf{\Omega}_\pi}f_W=\sum_{\pi \mathrm{~non-probl.}} \int_{\mathbf{\Omega}_\pi}f^-_U\cdot (1\pm d)\pm \sum_{\pi \mathrm{~non-probl.}}\frac{(e^\pi_{\mathrm{vert}}+e^\pi_{\mathrm{edge}})\ell}{\Delta^\ell}\;.
	\end{equation}
	We have
	\begin{align*}
		\sum_{\pi \mathrm{~non-probl.}}
		e^\pi_{\mathrm{vert}}&=\sum_{i=1}^\ell \sum_{\pi \mathrm{~non-probl.}} \mu\left(\pi(i)\cap \Lambda_{\mathrm{vert}}\right)\cdot \prod_{k\in V(T)\setminus \{i\}}\mu(\pi(k))\leByRef{eq:MEASURE2} \sum_{i=1}^\ell \mu\left(\Lambda_{\mathrm{vert}}\right)\leByRef{eq:boundsLambda} \ell\cdot a\;\mbox{, and}\\
		\sum_{\pi \mathrm{~non-probl.}}e^\pi_{\mathrm{edge}}&=
		\sum_{ij\in E(T)}
		\sum_{\pi \mathrm{~non-probl.}}
		\mu^{\otimes 2}\left(\Lambda_{\mathrm{edge}}\cap \bigcup_{ij\in E(T)} \pi(i)\times \pi(j)\right)\cdot \prod_{k\in V(T)\setminus \{i,j\}}\mu(\pi(k))\\
		&~\le |E(T)| \mu^{\otimes 2}\left(\Lambda_{\mathrm{edge}}\right)\leByRef{eq:boundsLambda} \ell \cdot 2a\;.
	\end{align*}
	Recall that in the course of proving Theorem~\ref{thm:main} from Theorem~\ref{thm:mainthm2} we showed that $\int_{\mathbold{\omega}\in \Omega}f_W(\mathbold{\omega})\le |\stab_T| \le \ell^\ell$. In particular, using the above bounds, we can turn the multiplicative error in~\eqref{eq:123} into an additive error,
	\begin{align}
		\begin{split}\label{eq:klice}
			\sum_{\pi \mathrm{~non-probl.}} \int_{\mathbf{\Omega}_\pi}f_W&=\sum_{\pi \mathrm{~non-probl.}} \int_{\mathbf{\Omega}_\pi}f^-_U \;\pm 2d\ell^\ell\pm  \frac{3a\ell^2}{\Delta^\ell}\\
			&=\sum_{\pi \mathrm{~non-probl.}} \int_{\mathbf{\Omega}_\pi}f^-_U
			\;\pm\frac{\epsilon}{100}\;.
		\end{split}
	\end{align}
	
	Let us now focus on the second sum in~\eqref{eq:expressW}. Observe that the set $\bigcup_{\pi \mathrm{~non-problematic}} \prod_{j=1}^\ell\pi(j)\setminus\mathbf{\Omega}_\pi$ has measure at most $\ell\cdot\mu(X)\le\delta_0$. Therefore,~\eqref{eq:Iwillgotobednow} applies, and we get that
	$$\sum_{\pi \mathrm{~non-problematic}}\int_{\prod_{j=1}^\ell \pi(j)\setminus\mathbf{\Omega}_\pi}f_W\le \frac{\epsilon}{10}\;.$$
	Let us now turn to the third sum in~\eqref{eq:expressW}. We write
	$$\sum_{\pi \mathrm{~not~dense~and~not~singular}} \int_{\prod_{j=1}^\ell \pi(j)}f_W
	\le
	\sum_{g=0}^{L}\sum_{\pi:\mathrm{witness~at~height}~g} \int_{\prod_{j=1}^\ell \pi(j)}f_W\;,$$
	where the last sum ranges over all non-singular depth preserving anatomy assignments $\pi$ with a witness for not being dense at height $g$. By Claim~\ref{cl:notdensenotsingular} and by~\eqref{eq:MEASURE1} and~\eqref{eq:MEASURE2} for each $g\in\{0,\ldots,L\}$ we have
	$$
	\sum_{\pi:\mathrm{witness~at~height}~g} \int_{\prod_{j=1}^\ell \pi(j)}f_W\le
	\frac{d\ell}{\Delta^\ell}
	+\frac{\ell^2}{\mathbf{h}_{L+1-g}\cdot\Delta^\ell}
	\le \frac{d'}{L+1}\;.
	$$
	Therefore,
	\begin{equation}\label{eq:ndns}
	\sum_{\pi \mathrm{~not~dense~and~not~singular}} \int_{\prod_{j=1}^\ell \pi(j)}f_W<d'\;.
	\end{equation} The fourth sum of~\eqref{eq:expressW} can be bounded from above as follows. If $\pi$ is singular then by Claim~\ref{cl:oneandall} there exists $i\in [\ell]$ such that we have $b_W(\omega)>1/(2\Delta)$ for all $\omega\in \pi(i)$, or we have $\pi(i)\subset \Omega_{\mathrm{small}}$. Since the average value of $b_W$ is~1 by Proposition~\ref{lem:avgb}, the measure of the elements that satisfy the former condition is at most $2\Delta$. Thus, the measure of the set $\bigcup_{\pi}\prod_{j=1}^\ell \pi(j)$, where the union ranges over all depth preserving anatomy assignments $\pi$ which are singular at coordinate $i$, is at most $2\Delta+\mu(\Omega_{\mathrm{small}})<\frac{\delta_0}{\ell}$. We conclude that the measure of $\bigcup_{\pi\mathrm{~singular}}\prod_{j=1}^\ell \pi(j)$ is less than $\delta_0$. Therefore,~\eqref{eq:Iwillgotobednow} applies, and $\sum_{\pi \mathrm{~singular}} \int_{\prod_{j=1}^\ell \pi(j)}f_W\le \frac{\epsilon}{10}$. It now remains to bound the second, third, and the fourth term~\eqref{eq:expressG}. Recall that when bounding the corresponding terms in~\eqref{eq:expressW} above, we argued that the domains of integration of the second and the fourth term have measures at most $\delta_0$ each. Combining this with the upper bound~\eqref{eq:fUbounded}, we get that the second and the fourth term are at most $\delta_0\cdot \frac{\ell}{\chi^\ell}<\frac{\epsilon}{100}$, using~\eqref{eq:bfh}. So, the only term we have to control now is
	\begin{align}\nonumber
		\sum_{\pi \mathrm{~not~dense~and~not~singular}} \int_{\prod_{j=1}^\ell \pi(j)}f^-_U=&\\
		\label{eq:keys}
		\sum_{\pi \mathrm{~not~dense~and~not~singular}} \int_{\mathbf{\Omega}_\pi}f^-_U&+\sum_{\pi \mathrm{~not~dense~and~not~singular}}\int_{\prod_{j=1}^\ell \pi(j)\setminus\mathbf{\Omega}_\pi}f^-_U\;.
	\end{align}
	The first summand can be bounded by $2d'$ in exactly the same way as we derived~\eqref{eq:ndns}. To see this, recall that on $\mathbf{\Omega}_\pi$ the degrees into anatomies are similar for $W$ and for $U$ and thus we can make use of the ``non-denseness'' and ``non-singularity'' assumption even for $f_U^-$ (with slightly worse constants).
	The second summand of~\eqref{eq:keys} corresponds represents integration over a set of measure at $\delta_0$. In other words, we can bound the second summand of~\eqref{eq:keys} by $\frac{\epsilon}{100}$ in the same way we bounded the second and the fourth term of~\eqref{eq:expressG}.
	
	The lemma now follows by expanding $\FREQ(T;W)-\FREQ^-(T;G,V_0,E_0)$ using~\eqref{eq:expressW} and~\eqref{eq:expressG} and using the bounds above.
\end{proof}

Theorem~\ref{thm:mainthm2} now follows in a straightforward way by combining Lemma~\ref{lem:congTquenched} and Lemma~\ref{lem:freqContinuous}.
\begin{proof}[Proof of Theorem~\ref{thm:mainthm2}]
	Suppose that a graphon $W$, an $\ell$-vertex tree $T$ and $\epsilon>0$ are given. Let $a>0$ by given by Lemma~\ref{lem:freqContinuous} for $W$, $T$, and $\frac{\epsilon}2$. Lemma~\ref{lem:furtherdecomp} with parameter $\beta:=c_1\cdot \min\{a^8,\epsilon^{16}\}$ (for a sufficiently small constant $c_1$) yields positive constants 
	$\alpha,\epsilon',\gamma$ and $\xi'$ with $\beta\gg \alpha \gg \epsilon'\gg \gamma \gg \xi'$. Now, let $n_0$ and $\delta'$ be numbers given by Lemma~\ref{lem:freqContinuous} for input parameters as above together with $\chi:=\gamma$.
	Set $\xi:=\min(\xi',\delta,\frac{1}{n_0})$.
	
	Suppose that $G$ is a graph with $d_\square(G,W)\leq \xi$. Owing to Lemma~\ref{lem:furtherdecomp}, we know that $G$ has a $(\beta,\alpha,\epsilon',\gamma)$-good decomposition $V(G)=V_0\sqcup V_1\sqcup\ldots\sqcup V_k$. Lemma~\ref{lem:congTquenched} now tells us that with probability at least $1-\LandauO(\ell \alpha^{1/8})$ the uniform spanning tree $\T$ of $G$ satisfies
	\begin{equation}\label{eq:rychle}
	\PROB_X(B_\T(X,r) \cong T) = (1+\LandauO(\ell^2 \alpha^{1/16})) \FREQ(T;G) + \LandauO( \ell^2 \beta^{1/8}) \, .
	\end{equation}
	Recall that the quantity $\FREQ(T;G)$ is defined in~\eqref{eq:mobil}, and depends on the decomposition $V(G)=V_0\sqcup V_1\sqcup\ldots \sqcup V_k$. Observe that $\FREQ(T;G)=\FREQ^-(T;G,U_0,E_0)$, where $U_0$ is the union of $V_0$ and sets $V_i$ ($i\ge 1$) that are not $(\alpha,\epsilon')$-big, and $E_0$ consists of all edges running across two distinct $(\alpha,\epsilon')$-big sets $\{V_i\}_i$. Proposition~\ref{prop:goodvis} tells us that $|U_0|= \LandauO(\beta^{1/8}n)<an$. The fact that $V(G)=V_0\sqcup V_1\sqcup\ldots \sqcup V_k$ is a $(\gamma,\epsilon'^5,\epsilon'^5)$-expander decomposition (in particular, (G2) of Definition~\ref{def:expanderdecomposition} applies) tells us that $|E_0|\le \epsilon'^5 n^2<an^2$. Thus, Lemma~\ref{lem:freqContinuous} gives $\FREQ(T;G)=\FREQ(T;W)\pm \frac{\epsilon}{2}$. Plugging this back to~\eqref{eq:rychle} and using that the  term $(1+\LandauO(\ell^2 \alpha^{1/16}))$ can be bounded by $(1\pm \frac{\epsilon}{4})$ and the term $\LandauO(\ell^2 \beta^{1/8})$ is smaller than $\frac{\epsilon}{4}$, we get the theorem.
\end{proof}

\subsection{Proof of Theorem~\ref{thm:main} from Theorem~\ref{thm:mainthm2}}\label{sec:Theorem1Theorem2}
Let $T$ be a fixed rooted tree with $\ell\geq 2$ vertices and height $r$. We denote the vertices of $T$, as before, by $\{1,\ldots,\ell\}$ so that $1$ is the root and the vertices of distance $r$ from the root are $\{p,\ldots,\ell\}$ for some $2 \leq p\leq \ell$. Our goal is to show that the probability that the first $r$ generations of the branching process yield a tree which is root-isomorphic to $T$ is $\FREQ(T;W)$.

First, the factor of $|\stab_T|^{-1}$ comes from the $|\stab_T|$ different vertex labellings of $T$ which yield the same event. Secondly, we may split this event to $\ell-p+1$ disjoint events, indexed by the vertices $q\in\{p,\ldots,\ell\}$ at height $r$ of $T$, and each of these is the event that the path from $1$ to $q$ is the first $r$ steps of the ancestral path defined in $\kappa_W$. We will show that the probability of each of these events is precisely
\begin{equation}\label{eq:qterm} \int\displaylimits_{\omega_1, \ldots, \omega_\ell} \exp \left(-\sum_{j=1}^{p-1} b_W(\omega_j) \right)\cdot \frac{\deg(\omega_q) }{\prod_{j=1}^\ell \deg(\omega_j)} \cdot\prod_{\substack{(i,j) \in E(T)}} W(\omega_i, \omega_j) \mathrm{d}\omega_1 \cdots \mathrm{d}\omega_\ell \, ,\end{equation}
which is the $q$-th term in \eqref{eq:frewTW} when we expand the sum over $j$ in the numerator. We denote the path from $1$ to $q$ in $T$ by $x_1,\ldots, x_r$ where $x_1=1$ and $x_r=q$. By definition of $\kappa_W$, the density function of the first $r$ ancestral particles $\omega_{x_1},\ldots, \omega_{x_r}$ is
$$ \prod_{j=1}^{q-1} {W(\omega_{x_j}, \omega_{x_{j+1}}) \over \deg(x_j)} \, .$$
Thirdly, by independence of the branching process, conditioned on this path, the rest of the branching process emanating from the particles on this path is independent and all new particles (if there are any) are ``other'' particles. Given a particle of type $\omega$ (either ancestral or other), the number of its other progeny is distributed as $\POISSON(b_W(w))$ random variable, so the probability that it equals $k \geq 0$ is ${e^{-b_W(\omega)} b_W(\omega)^k \over k!}$. Given the the number of its ``other'' progeny is $k$, these $k$ points of $\Omega$ are distributed as $k$ i.i.d. points drawn from the probability density function ${W(\omega,\omega') \over b_W(\omega) \deg(\omega')}$. Thus, conditioned on having $k$ progeny, the density of these $k$-tuple $\{\omega_1, \ldots, \omega_k\}$ of points is
$$ {k! \over b(\omega)^k} \prod_{i=1}^k { W(\omega, \omega_i) \over \deg(\omega_i)}  \, .$$
The $k!/b_W(\omega)^k$ cancels with the $b_W(\omega)^k/k!$ in the probability of having $k$ points and we are left with last term times $e^{-b_W(\omega)}$. We now apply this throughout all the branching points of the tree $T$ (even for $k=0$) and this concludes the proof (note that we do not have a factor for $e^{-b_W(\omega)}$ for tree vertices at level $r$ since we do not care if they have progeny or not). \qed

\section{Bounds on the number of vertices of a given
degree in the UST}\label{sec:degreeextremal}
In this section we prove Theorem~\ref{cor.degdist}. Let us make two remarks beforehand. Firstly, we cannot hope for converse bounds to the theorem. Indeed, let us take a small $\alpha>0$ and let us consider the complete bipartite graph $K_{\alpha n,(1-\alpha n)}$ with color classes $A$ and $B$, $|A|=\alpha n$, $|B|=(1-\alpha)n$.
The handshaking lemma tells us that for \emph{any} spanning tree $T$ of $K_{\alpha n,(1-\alpha n)}$ we have
$$n-1=\sum_{b\in B}\deg_T(b)=(1-\alpha)n+\sum_{b\in B,\deg_T(b)>1}(\deg_T(b)-1)\;.$$
In particular, the last sum can have at most $\alpha n-1$ summands. We conclude that $T$ has more than $(1-2\alpha )n$ leaves, and less than $2\alpha n$ vertices of degrees more than 1.

Secondly, as we mentioned in the beginning of this paper, the degree distribution of a UST in a complete graph is approximately $\POISSON(1)+1$. Considering Theorem~\ref{cor.degdist}, we see that the complete graph is an asymptotic minimizer for the number of leaves (density $e^{-1}$), the asymptotic maximizer for the number of vertices of degree $2$ (density $e^{-1}$) and $3$ (density $e^{-1}/2$). However, the density of vertices degree $k \geq 4$ in the UST of the complete graph is $e^{-(k-1)}/(k-1)!$ which is smaller than the $(k-2)^{k-2}e^{-(k-2)}/(k-1)!$ given by Theorem~\ref{cor.degdist}. Examining the proof shows that the bounds in Theorem~\ref{cor.degdist} for $k \geq 4$ are optimal in the following sense. Let $\alpha>0$ be arbitrarily small, and let $G_{n,k,\alpha}$ be the graph obtained by taking a complete graph on $n/(k-2)$ vertices and additional $n(k-3)/(k-2)$ vertices such that each of the $n^2(k-3)$ possible edges between the two parts is retained with probability $\alpha$ and erased otherwise, independently of all other edges. It is a straightforward computation using Theorem~\ref{thm:main} (see also the proof of Theorem~\ref{cor.degdist} below) to show that for $n$ large, with high probability the proportion of vertices of degree $k$ in a UST will be arbitrarily close, as $\alpha \to 0$, to $\frac{1}{(k-1)!}((k-2)/e)^{k-2}$.

We are almost ready to prove the theorem. We will need the following folklore  statement that provides a relation between degrees in a graphon and finite graphs that converge to it. Lemma~\ref{lem:graphondegrees} is easy to prove directly, but let us note that is a straightforward consequence of Lemma~\ref{lem:anatomiescontNEW} for anatomies of depth~$1$.
\begin{lem}
	\label{lem:graphondegrees}
	Suppose that $G_1,G_2,\ldots$
	are finite graphs converging to a graphon $W:\Omega^2\rightarrow[0,1]$.
	\begin{enumerate}
		\item\label{en:alldeg} Suppose that $0\le a<b\le 1$ are given. and let $d$ be the measure of points of $\Omega$ whose degree is in the interval $(a,b)$. Then for an arbitrary $\epsilon>0$ there exists an $n_0$ such that for every $n>n_0$ we have that in $G_n$ there are at least $(d-\epsilon) v(G_n)$ many vertices with degrees in the interval $(  (a-\epsilon)v(G_n)  ,   (b+\epsilon) v(G_n) )$ and there are at most $(d+\epsilon) v(G_n)$ many vertices with degrees in the interval $(  (a+\epsilon)v(G_n)  ,   (b-\epsilon) v(G_n) )$.
		\item\label{en:mindeg} Suppose that there exists $\alpha>0$ such that in each $G_n$, all but at most $\Landauo_n(1)v(G_n)$ vertices have degrees at least $\alpha v(G_{i})$. Then $W$ is nondegenerate.
	\end{enumerate}
\end{lem}

\medskip

\begin{proof}[Proof of Theorem~\ref{cor.degdist}] We start by proving~\eqref{eq:extdeg1}. Assume by contradiction that there are $\eps_0, \delta_0>0$ such that there exists a sequence of graphs $G_n$ (which we assume have $n$ vertices) such that
	\begin{enumerate}
		\item[(A)] at least $(1-\Landauo(1))n$ vertices of $G_n$ are of degree at least $\delta_0 n$, and,
		\item[(B)] $\PROB\big ( L_1(n) \leq (e^{-1}-\eps_0)n \big ) > \eps_0$,
	\end{enumerate}
	where $L_1(n)$ is the number of leaves in a UST of $G_n$. By Theorem~\ref{thm:graphsconverge} we may assume without loss of generality that $G_n$ is a converging sequence and let $W$ be the corresponding limit graphon. By Lemma~\ref{lem:graphondegrees}\eqref{en:mindeg} and assumption (A) we deduce that $W$ is nondegenerate. Let $\T_n$ be a UST of $G_n$. By Theorem~\ref{thm:main} the sequence $\T_n$ almost surely satisfies that
	$$ \lim_{n\to\infty} \PROB( \deg_{\T_n}(X) = 1 ) = \EXP_{x\in\Omega}\left[\PROB[1+\POISSON(b_W(x))=1]\right]=\EXP_{x\in\Omega}\left[e^{-b_W(x)}\right] \, ,$$
	where $X$ is a uniformly chosen vertex of $G_n$. By Proposition \ref{lem:avgb} we have that $\EXP_{x\in\Omega}[b_W(x)]=1$. Hence Jensen's inequality implies that
	$$ \lim_{n \to \infty} \PROB( \deg_{\T_n}(X) = 1 ) \geq e^{-1} \, .$$
	Since $L_1(n) = n\PROB( \deg_{\T_n}(X) = 1 )$ we arrive at a contradiction to assumption (B), showing~\eqref{eq:extdeg1}. \\
	
	The proof of~\eqref{eq:extdeg2} and~\eqref{eq:extdegbiggerthan2} proceeds similarly. We fix $k \geq 2$ and make the an analogous contradictory assumption as in (A) and (B).
	We have that
	$$ \lim_{n\to\infty} \PROB( \deg_{\T_n}(X) = k ) = \EXP_{x\in\Omega}\left[\PROB[1+\POISSON(b_W(x))=k]\right]={1 \over (k-1)!}\EXP_{x\in\Omega}\left[e^{-b_W(x)} b_W(x)^{k-1} \right] \, .$$
	When $k=2$, since $ye^{-y}\leq e^{-1}$ for all $y \geq 0$ we learn that the last quantity is at most $e^{-1}$, proving~\eqref{eq:extdeg2} by contradiction. When $k \geq 3$, it follows by Lemma \ref{lem:optimize} (provided immediately below) that the last quantity is at most ${1 \over (k-1)!}((k-2)/e)^{k-2} $, similarly proving~\eqref{eq:extdegbiggerthan2}.
\end{proof}

The last missing piece is to state and prove the optimization result that was key for our proof of Theorem~\ref{cor.degdist} above.
\begin{lem}\label{lem:optimize} Let $b:[0,1]\to [0,\infty)$ satisfy $\int_{[0,1]} b(x)\mathrm{d}x =1$. Then for any $k \geq 2$,
	$$ \int_{[0,1]} e^{-b(x)} b(x)^k \mathrm{d}x \leq \left ( {k-1 \over e} \right )^{k-1}.$$
\end{lem}
\begin{proof}
	For ease of notation, let $f(x)=e^{-x}x^k$. Fix $\kappa>0$. Approximating $b(x)$ by simple functions, we can find $n\in \bN$ and $n$ non-negative numbers $b_1,\ldots,b_n$ such that $\frac{1}{n}\sum_{i=1}^{n}b_i=1$, and $\int_{[0,1]} f(b(x))\mathrm{d}x \le \frac{1}{n}\sum_{i=1}^{n}f(b_i)+\kappa$. Hence
	\begin{equation*}
		\int_{[0,1]} e^{-b(x)} b(x)^k \mathrm{d}x \leq \max_{(b_1,\ldots,b_n)\in \cC} \Phi(b_1,\ldots,b_n) + \kappa,
	\end{equation*}
	where $\Phi(b_1,\ldots,b_n):=\frac{1}{n}\sum_{i=1}^{n}f(b_i)$ and $\cC:=\{(b_1,\ldots,b_n)\in [0,+\infty)^{n}:\frac{1}{n}\sum_{i=1}^{n}b_i \le 1\}$. To prove the lemma, it thus suffices to show that $\Phi \vert_{\cC} \le \left(\frac{k-1}{e}\right)^{k-1}$.
	
	By abuse of notation, we use the same symbol $\vec{b}=(b_1,\ldots,b_n)$ to denote the maximizer of $\Phi\vert_{\cC}$. As $f'(x)=e^{-x}x^{k-1}(k-x)<0$ for every $x>k$, $f(x)$ is  decreasing on $(k,\infty)$, and consequently
	\begin{equation}
	\label{eq:bound-b}
	0 \le b_i \le k \quad \text{for all $i\in [n]$}.
	\end{equation}
	There are two possibilities.
	
	{\bf Case 1:} There exists a non-negative number $y$ so that $b_i \in \{0,y\}$ for every $i\in [n]$.\\
	Let $\lambda \in [0,1]$ denote the proportion of $i\in [n]$ with $b_i=y$. Then $\lambda y=\frac{1}{n}\sum_{i=1}^{n}b_i \le 1$. Hence
	\begin{equation*}
		\Phi(\vec{b})=\lambda e^{-y}y^k \le e^{-y}y^{k-1} \le \left(\frac{k-1}{e}\right)^{k-1},
	\end{equation*}
	as required.
	
	{\bf Case 2:} $b_i$ receives at least two positive values.\\
	Let $i$ and $j$ be two arbitrary indices such that $1\le i<j \le n$ and $\min \{b_i,b_j\}>0$. For every small $\eps>0$, two vectors $\vec{b}':=(b_1,\ldots,b_i+\eps,\ldots,b_j-\eps,\ldots,b_n)$ and $\vec{b}'':=(b_1,\ldots,b_i-\eps,\ldots,b_j+\eps,\ldots,b_n)$ clearly belong to the domain $\cC$. By Taylor's formula and the optimality of $\vec{b}$, we thus obtain
	\begin{align*}
		0 \le \Phi(\vec{b})-\Phi(\vec{b}')&=(f'(b_j)-f'(b_i))\eps +O_{b_i,b_j}(\eps^2),\\
		0 \le \Phi(\vec{b})-\Phi(\vec{b}'')&=-(f'(b_j)-f'(b_i))\eps +O_{b_i,b_j}(\eps^2).
	\end{align*}	
	This forces $f'(b_i)=f'(b_j)$.
	
	As $f''(x)=e^{-x}x^{k-2}((x-k)^2-k)$, we see that $f''(x)>0$ for $x\in (0,k-\sqrt{k})$, and $f''(x)<0$ for $x\in (k-\sqrt{k},k)$. It follows that $f'(x)$ is increasing in $[0,k-\sqrt{k}]$, and decreasing in $[k-\sqrt{k},k]$. Hence for each $\alpha \in \bR$, the equation $f'(x)=\alpha$ has at most two solutions in $[0,k]$.
	
	From the discussion above and \eqref{eq:bound-b}, we learn that there are two numbers $y$ and $z$ satisfying the following properties:
	\begin{itemize}
		\item[\rm (i)] $b_i \in \{0,y,z\}$ for every $i \in [n]$,
		\item[\rm (ii)] $f'(y)=f'(z)$, that is, $e^{-y}y^{k-1}(k-y)=e^{-z}z^{k-1}(k-z)$,
		\item[\rm (iii)] $0<y<z \le k$.
	\end{itemize}
	Let $\lambda$ and $\mu$ be the proportions of $i \in [n]$ such that $b_i=y$ and $b_i=z$, respectively. Then $\lambda+\mu \le 1$, $\lambda y+\mu z\overset{(i)}{=}\frac{1}{n}\sum_{i=1}^{n}b_i \le 1$, and $\Phi(\vec{b})\overset{(i)}{=}\Psi(\lambda,\mu)$, where $\Psi(\lambda,\mu):=\lambda f(y)+\mu f(z)$.
	
	We will view $y$ and $z$ as constants, and seek to maximize $\Psi(\lambda,\mu)$ under the constraints: $\lambda,\mu \ge 0$, $\lambda+\mu \le 1$, and $\lambda y+\mu z \le 1$. Let $(\lambda_0,\mu_0)$ be a maximizer of $\Psi$. We claim that $\min\{\lambda_0,\mu_0\}=0$. Before proving the claim, let us show how it implies the lemma. Indeed, if one of $\lambda_0$ and $\mu_0$ is zero, say $\mu_0$, then
	\[
	\Phi(\vec{b})=\Psi(\lambda,\mu) \le \Psi(\lambda_0,\mu_0)=\lambda_0 e^{-y}y^k \le e^{-y}y^{k-1} \le \left(\frac{k-1}{e}\right)^{k-1},
	\]
	as desired.
	
	It remains to prove that $\min\{\lambda_0,\mu_0\}=0$. Suppose to the contrary that $\lambda_0,\mu_0>0$. Let $\lambda_1=\lambda_0-\epsilon$ and $\mu_1=\mu_1+\epsilon y/z$, where $\epsilon>0$ is a small constant. It is not difficult to verify that $\lambda_1,\mu_1 \ge 0$, $\lambda_1+\mu_1 \le 1$, and $\lambda_1 y + \mu_1 z \le 1$. Thus, by the optimality of $(\lambda_0,\mu_0)$, we get
	$\Psi(\lambda_1,\mu_1) \le \Psi(\lambda_0,\mu_0)$, giving
	$e^{-z}z^{k-1} \le e^{-y}y^{k-1}$. Combined with (iii), we obtain $e^{-z}z^{k-1}(k-z)<e^{-y}y^{k-1}(k-y)$, contradicting (ii). This completes our proof of the lemma.
\end{proof}

\begin{remark}
Our proof of Theorem~\ref{cor.degdist} crucially relied on our running assumption that almost all the degrees in the graph $G$ are linear. It seems natural to investigate similar extremal questions with a weakened form of this assumption. The minimal meaningful assumption seems to be that $G$ contains no vertices of degree $2$; this is to avoid the case of paths. This general setting seems much more complicated. For example, the proportion of leaves in a uniform spanning tree on an $n\times n$ torus is asymptotically almost surely $(1-2/\pi)\cdot\frac8{\pi^2}+\Landauo(1)\approx 0.294$ (see for example \cite[p.~112]{LyPe:ProbabilityTrees}), which is less than the lower bound of $e^{-1}\approx 0.368$ in Theorem~\ref{cor.degdist}.
\end{remark}

\section{Acknowledgments}
Hladk\'y and Tran were supported by the Czech Science Foundation, grant number GJ16-07822Y and by institutional support RVO:67985807. Hladk\'y is supported by the Alexander von Humboldt Foundation. Nachmias is supported by ISF grant 1207/15, and ERC starting grant 676970 RANDGEOM. Part of this work was done while Hladk\'y and Tran were visiting Tel Aviv University. They thank the department of mathematical sciences at TAU for their hospitality. The trip was funded by the Czech Academy of Sciences.

We thank Omer Angel and Tom Hutchcroft for their assistance with the proof of Lemma \ref{lem:revisitexpander} and to Michael Krivelevich and Wojciech Samotij for useful discussions.

\bibliographystyle{plain}
\bibliography{bibl}

\end{document}